\documentclass{article}

\usepackage[utf8]{inputenc}
\usepackage{hyperref}
\hypersetup{
pdftitle={Lyapunov exponents and switching systems},
pdfauthor={Andrei A. Agrachev,Michele Motta},
}

\usepackage[english]{babel}
\usepackage{amsmath}
\usepackage{amssymb}
\usepackage{amsthm}
\usepackage{amsfonts}
\usepackage{mathtools}
\usepackage{makeidx}
\usepackage{etoolbox}
\usepackage{bbm}
\usepackage{subfiles}
\usepackage{graphicx}
\usepackage{subfigure}
\usepackage{svg}

\usepackage{pgfplots}
\usepackage{tikz}
\usepackage{tikz-3dplot}

\usepackage{xcolor}
\usepackage{pdflscape}

\usepackage[left=3cm,right=3cm]{geometry}

\usepackage[backend=biber]{biblatex}
\usepackage{csquotes}
\bibliography{biblio.bib}

\newtheorem{thm}{Theorem}[section]
\newtheorem{cor}[thm]{Corollary}
\newtheorem{prop}[thm]{Proposition}
\newtheorem{lem}[thm]{Lemma}

\newtheorem*{prob}{Problem}
\AfterEndEnvironment{prob}{\noindent\ignorespaces}

\theoremstyle{definition}
\newtheorem{defn}[thm]{Definition}

\theoremstyle{remark}
\newtheorem{rem}[thm]{Remark}

\pgfplotsset{compat=1.15}
\usetikzlibrary{3d}
\usetikzlibrary{shapes}

\pgfplotsset{compat=newest}

\bibliography{biblio.bib}

\makeatletter
\let\c@equation\c@thm
\makeatother
\numberwithin{equation}{section}

\newcommand{\C}{\mathbb{C}}
\newcommand{\R}{\mathbb{R}}

\newcommand{\arcosh}{\mathrm{arcosh}}
\newcommand{\al}{\alpha}
\newcommand{\be}{\beta}
\newcommand{\e}{\eta}

\newcommand{\Ham}{\mathcal{H}}

\newcommand{\sld}{\mathfrak{sl}_2(\R)}
\newcommand{\tr}{\mathrm{tr}}
\newcommand{\SLD}{\mathrm{SL}_2(\R)}
\newcommand{\D}{\mathrm{d}}
\newcommand{\I}{\mathrm{Id}}

\newcommand{\arctanh}{\mathrm{arctanh}}

\newcommand{\pa}{\partial}

\newcommand{\la}{\lambda}

\newcommand{\U}{\mathcal{U}}
\newcommand{\A}{\mathcal{A}}

\newcommand{\tu}{\Tilde{u}}

\newcommand{\vp}{\varphi}
\newcommand{\dvp}{\dot{\varphi}}

\newcommand{\de}{\dot{\eta}}

\newcommand{\calM}{\mathcal{M}}
\newcommand\numberthis{\addtocounter{equation}{1}\tag{\theequation}}

\newcommand\ddfrac[2]{\frac{\displaystyle #1}{\displaystyle #2}}

\title{Lyapunov Exponents of Linear Switched Systems}

\author{Andrei A. Agrachev, Michele Motta \thanks{SISSA, Trieste}}

\begin{document}

\maketitle

\begin{abstract}
\noindent
We explicitly compute the maximal Lyapunov exponent for a switched system on $\mathrm{SL}_2(\mathbb R)$. This computation is reduced to the characterization of optimal trajectories for an optimal control problem on the Lie group.
\\[10pt]
\noindent \textbf{Keywords.} Control Theory, Optimal Control, Switched system, Dynamical Systems, Lyapunov Exponents 
\\

\end{abstract}

\tableofcontents

\newpage

\begin{section}{Introduction}
    Let
    \begin{equation}
    \dot y=A(t)y,\quad y\in\mathbb R^n, 
    \end{equation}
    be a linear system of ordinary differential equations; here $t\mapsto A(t),\ t\ge 0,$ is a measurable bounded family of $n\times n$-matrices. Let $X(t)$ be the fundamental matrix of system (1.1), the solutions of (1.1) have a form:
    $y(t)=X(t)y(0)$.
    
    Principal Lyapunov exponent of system (1.1) is defined as follows:
    \begin{equation}
    \label{eq1}
    \ell(A(\cdot))=\limsup\limits_{t\to\infty}\frac 1t\ln\|X(t)\|.
    \end{equation}

    This quantity does not depend on the choice of norm in the space of matrices and is a natural measure of instability of system (1.1).
    Indeed, system (1.1) is asymptotically stable with an exponential convergence rate if and only if $\ell(A(\cdot))<0$.

    Principal Lyapunov exponents play a key role in the general theory of Dynamical Systems, see \cite{Viana}, \cite{Pesin}.
    
    If $A(t)\equiv A$ is a constant matrix, then $\ell(A)$ is maximum of the real parts of the eigenvalues of $A$. Moreover, for any $A(\cdot)$ and any $c\in\mathbb R$, we have $\ell(A(\cdot)+cI)=\ell(A(\cdot))+c$.
    
    In what follows, we use standard notations: $\mathfrak{gl}_n(\mathbb R)$ is Lie algebra of all real $n\times n$-matrices and
    $\mathrm{GL}_n(\mathbb R)$ is Lie group of all nondegenerate real $n\times n$-matrices. Similarly $\mathfrak{sl}_n(\mathbb R)$
    is Lie algebra of all real $n\times n$-matrices with zero trace and $\mathrm{SL}_n(\mathbb R)$ is Lie group of all real $n\times n$-matrices whose determinant equals 1.
    
    Let $S\subset\mathfrak{sl}_n(\mathbb R)$ be a compact subset. A switched system induced by $S$ is the set of systems of ordinary differential equations
    $$
    \dot y=A(t)y,\quad A(t)\in S,\ 0\le t<\infty,
    $$
    where $t\mapsto A(t),\ t\ge 0,$ is any measurable map with values in $S$.
    
    Let $\mathcal A_S\subset L_\infty([0,\infty);\mathrm{GL}_n\mathbb R))$ be the set of all such maps. Principal Lyapunov
    exponent of the switched system is defined as follows:
    $$
    \ell(S)=\sup\limits_{A(\cdot)\in\mathcal A_S}\ell(A(\cdot))\,.
    $$
    Clearly, $\ell(S)\ge\sup\limits_{A\in S}\ell(A)$\,.
    
    We say that the switched system is exponentially stable if $\ell(S)<0$. Stability of switched dynamical systems is really important for applications and there is a big literature devoted to this subject, see for instance \cite{AgLi}, \cite{LibSwSyst}, \cite{Lib2023},
    and references there.
    
    A more natural mathematical problem is to compute or at least to estimate principal Lyapunov exponent of any switched system; many results about stability can be naturally generalized in this direction.
    
    In particular, paper \cite{AgLi} is about stability but if you analyse proofs of the main results you will find the characterization of Lie subalgebras
    $\mathfrak g\subset\mathfrak{gl}_n(\mathbb R)$ such that
    \begin{equation}
    \label{Lyaaap}
            \ell (S)
            =
            \sup
            \limits_{A\in S}
            \ell(A),
            \quad 
            \forall
            \, 
            S
            \subset 
            \mathfrak 
            g
            \,
            .
    \end{equation}

    Namely, a Lie subalgebra $\mathfrak g\subset\mathfrak{gl}_n(\mathbb R)$ has property \eqref{Lyaaap} if and only if $\mathfrak g$ does not contain
    a subalgebra isomorphic to $\mathfrak{sl}_2(\mathbb R)$.
    
    In the current paper, we explicitly compute $\ell(\{A,B\})$\footnote{In what follows we use shortened notation
    $\ell(A,B)\doteq\ell(\{A,B\})$} for any pair of matrices $A,B$ which generates Lie algebra $\mathfrak{sl}_2(\mathbb R)$.
    Such a formula automatically provides us with an explicit expression of $\ell(A,B)$ for matrices $A,B$ of any size if
    these matrices generate a Lie algebra isomorphic to $\mathfrak{sl}_2(\mathbb R)$.
    
    Indeed, let $\Psi:\mathfrak{sl}_2(\mathbb R)\to Lie(A,B)$ be such isomorphism; then $\Psi$ is a representation of $\mathfrak{sl}_2(\mathbb R)$. The representation $\Psi$ is a direct sum of irreducible representations,
    $\Psi=\bigoplus\limits_{i=1}^k\Psi_i$, where $\Psi_i:\mathfrak{sl}_2(\mathbb R)\to\mathfrak{sl}_{n_i}(\mathbb R),$ $i=1,\ldots,k,$
     $n_1\ge\cdots\ge n_k\ge 2$. Then
     $$
     \ell(A,B)=(n_1-1)\ell(\Psi^{-1}(A),\Psi^{-1}(B)).
     $$
    
     We hope that our computation will serve as a building block for eventual estimates of the Lyapunov exponents in much more general cases.
\\
    Our final result can be stated as follows. To simplify the notation, set $a=\tr(A^2),b=\tr(B^2),c=\tr(AB)$.
    \begin{thm}[Main Theorem]
        Suppose that $A,B\in\sld$, with $a\geq b$. 
        \begin{enumerate}
            \item If $a\geq 0$ and $ c>a$ or $a<0$ and $ c\geq\sqrt{a b}$, then 
            \begin{equation*}
                \ell(A,B)
                =
                \frac{1}{2}                
                \sqrt{
                    \frac{
                            c^2 - ab
                        }
                    {
                          c-\frac{a+b}{2}
                    }
                };
            \end{equation*} 
            \item If $a\geq 0$ and $ c\leq a$, then 
            \begin{equation*}
                \ell(A,B)
                =
                \sqrt{
                    \frac{a}{2}
                };
            \end{equation*}
            \item If $a<0$ and $ c\leq-\sqrt{a b}$, then 
            \begin{equation*}
                \ell(A,B)
                =
                \frac{2}{\pi}
                \frac{
                    1
                }{
                    \sqrt{\frac{-2}{a}}
                    +
                    3\sqrt{\frac{-2}{ b}}
                }
                \arcosh
                \left(
                    \frac{- c}{\sqrt{a b}}
                \right).
            \end{equation*}
        \end{enumerate}
    \end{thm}
    We point out that for every $A,B\in\sld$ such that $a, b\leq 0$, then $ c^2\geq a b$ (see Section 6 for more explanations).
    \\
    In Figure \ref{GraphCost} you can see a graphical representation of Main Theorem.
    \\
    The whole paper is devoted to prove our Main Theorem and it is structured as follows:
    in Section 2 we show how computing the Lyapunov exponents $\ell (A,B)$ can be reduced to an optimal control problem and we will apply Pontryagin Maximum Principle (PMP) to this problem. We stated the precise version of PMP that we used in the Appendix.
    In Section 3 we will deduce from PMP some general properties of Pontryagin extremals. In particular, we will divide them into two categories: \emph{bang-bang} extremals and singular extremals. 
    In Section 4 and 5 we study \emph{bang-bang} and singular extremals respectively. 
    In Section 6 we show which extremal is the optimal one, depending on $\tr(A^2),\tr(B^2)$ and $\tr(AB)$. 
    \begin{figure}
        \centering
        \includegraphics[scale=0.6]{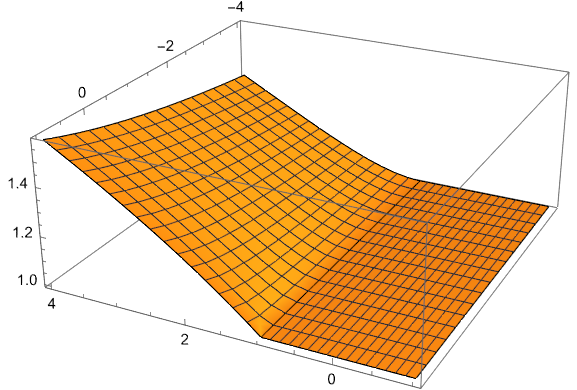}
        \includegraphics[scale=0.6]{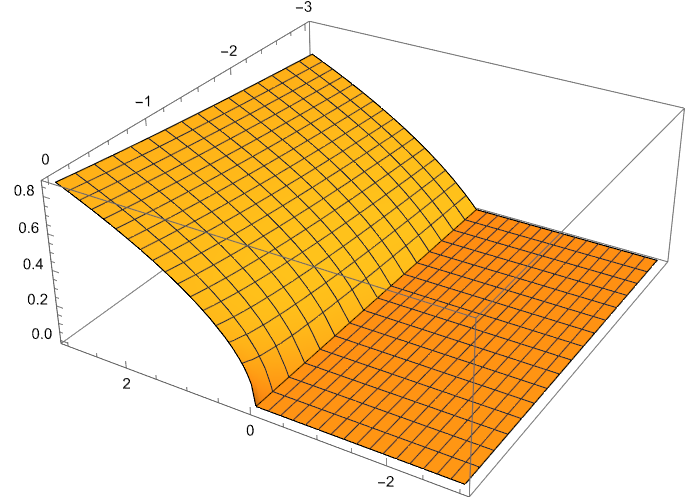}
        \includegraphics[scale=0.6]{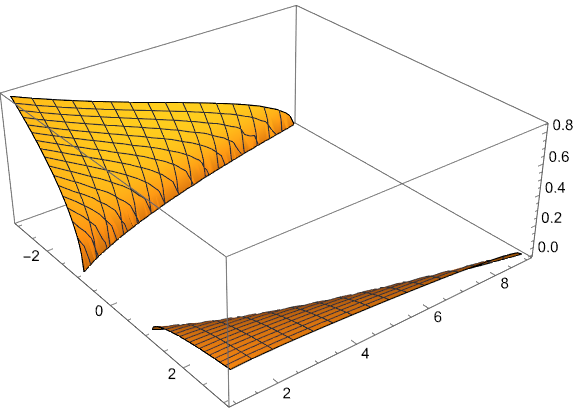}
        \caption{Graphics of $\ell(A,B)$ for different ranges of $\tr(A^2),\tr(B^2),\tr(AB)$. In the first picture $\tr(A^2)$ is normalized to $1$, while $\tr(B^2)\leq 1$ and $\tr(AB)\in\R$; in the second one $\tr(A^2)=0$, $\tr(B^2)\leq 0$ and $\tr(AB)\in\R$;
        in the third one $\tr(A^2)$ is normalized to $-1$, while $\tr(B^2)\geq 1$ and $|\tr(AB)|\geq -\sqrt{\tr(B^2)}$.}
        \label{GraphCost}
    \end{figure}
\end{section}

\begin{section}{Reduction to OC Problem and application of PMP}
    As explained in the Introduction, we will study the switched system induced by two matrices $A,B$:
     \begin{equation}
        \label{eq2}
        \left\{
        \begin{array}{l}
             \dot{X}(t)
             =
             X(t) 
             \big( 
             u(t) A 
             + 
             (1-u(t)) B 
             \big),
             \\
             X(0)=\I,              
        \end{array}
        \right.
    \end{equation}
    where 
    $
    X : [0,T] \to \SLD
    $
    , $T>0$,
    $A,B\in\sld$, such that $A,B,[A,B]$ are linearly independent, 
    $u\in\mathcal{U}:=\{v:[0,T]\to\{0,1\} \; | \; v \text{ measurable} \}$. Without loss of generality, we can suppose $\tr(A^2)\geq\tr(B^2)$. 
    \\
    Notice that here we write $\dot{X}=XA(t)$, while in the Introduction we had $\dot{X}=A(t)X$. Indeed, using the notation of the Introduction, system \eqref{eq2} should be written as $\dot{X}^T=X^T A(t)^T$, where the $T$ superscript denotes the transpose matrix. To simplify the notation we will drop this superscript.
    \\
    We prefer to solve the system with $A$ and $X$ in the reversed order because system \eqref{eq2} is left-invariant.
    \\
    We want to find the principal Lyapunov exponent $\ell(A,B)$. 
    Fixed $u\in \U$, we will denote with $X(\cdot;u)$ the solution of \eqref{eq2} with control equal to $u$. 
    \\
    A natural relaxation of our problem \eqref{eq2} consists in taking $u : [0,T]\to [0,1]$ measurable instead of $u : [0,T]\to \{0,1\}$. 
    This way, we can restate the original problem into a control problem and so we can use all the classical results from Control Theory. 
    \\
    Given $X\in\SLD$ with real eigenvalues, let us denote with $\la(X)$ the eigenvalue of $X$ bigger than 1. If $X$ has complex eigenvalue, we denote with $\la(X)$ the real part of these eigenvalues.
    \\
    Another crucial observation is that if the norm in \eqref{eq1} is the norm induced by the Euclidean metric on $\R^2$, then its value coincides with the absolute value of $\la(X(T))$. 
    Hence,
    if for every $T>0$ we can compute $c(T)$ such that 
    \begin{equation*}
        \sup_{u\in\U}
        \la (X(T;u)) 
        = 
        c(T), 
    \end{equation*}
    then, we obtain that
    \begin{equation*}
        \ell(A,B)
	   = 
	\limsup_{T\to+\infty}
        \frac{1}{T} 
        \ln c(T).
    \end{equation*}
    To be precise, the previous equality should be a ``lower or equal" a priori, because the optimal strategy to maximize $\la$ up to time $T$ can be change when $T$ goes to infinity. 
    We prove however that the limit is realized by a fixed control $u(t)$, which is explicitly computed. 
    For those familiar with the terminology of Control Theory, the optimal strategy will be one of the following three possibilities: 
    the constant singular control, corresponding to case 1. in the Main Theorem,
    the constant control equal to 1, corresponding to case 2.,  
    a bang-bang periodic control, corresponding to case 3. 
    We will define these concepts later.
    \\[3pt]
    Instead of finding 
    $
    \sup_u 
    \la (
        X(T;u)
    )
    $, one can find 
    $ 
    \sup_u 
    f
    \big(
        \la (
            X(T;u)
        )
    \big)
    $, 
    for a suitable $f$ monotone in $\la$.
    It turns out that a convenient choice of $f$ is the trace $\tr(X(T;u))$: besides being monotone in $\la$, it is a linear function of $X$ and it is invariant after a change of basis. 
    \\
    Thus, we can restate our original problem into the following Optimal Control problem:
    \begin{prob}
        Fix $A,B\in\sld$ and $T\geq 0$. Let $\mathcal{U}=\{u:[0,T]\to[0,1] \; | \; u \text{ measurable} \}$ and for each $u\in\U$ denote with $X(\cdot;u)$ the solution of the following Cauchy problem on $\SLD$:
        \begin{equation}
		\label{Group_eq}
		\begin{cases}
				\dot{X}(t)=X(t)(u(t)A+(1-u(t))B), \\
				X(0)=\I\in\SLD,
		\end{cases} \quad t\in[0,T].   
	\end{equation}
	Define the functional 
        $\Phi_T:\mathcal{U}\to \R$, 
        $\Phi_T(u)=\tr X(T;u)$. 
        We want to find 
        $\tu\in\mathcal{U}$ 
        which maximizes $\Phi_T$:
	\begin{equation}
            \label{max_prob}
		\tr 
            X(T;\Tilde{u})
            =
            \max_{u\in\mathcal{U}} 
            \tr X(T;u).
	\end{equation}
    \end{prob}
    If we are able to solve this problem for each fixed $T$, then from the estimate of 
    $\tr X(T;u)$ we can deduce an estimate for $\|X(T;u)\|$ which is uniform with respect to $u\in\U$ and depends only on $T$. 
    Then, the desired value of the principal Lyapunov exponent follows.
    \\
    Notice that if $[A,B]$ is a linear combination of $A$ and $B$, then the Lie algebra generated by $A,B$ is two-dimensional, hence is solvable. In this case we already know from \cite{AgLi} the value of $\ell(A,B)$. 
%
%
%
%
    \\[5pt]
    Existence of maximum in \eqref{max_prob} is guaranteed by Filippov Theorem (see, for instance, Theorem 10.1 in \cite{AgSa}).
    \\
    So, we can apply Pontryagin Maximum Principle (PMP) to our problem. For a general reference, see Theorem 12.3 in \cite{AgSa}. You can find in Appendix A (see Theorem \ref{PMP}) the precise statement that we use here.
    \\
    We define the Pontryagin left-invariant Hamiltonian function:
    \begin{equation}
        H(\e,u)
        =
        \tr\big(
            \e(
                uA+(1-u)B
                )
            \big)
        \quad 
        \e\in\sld,
        u\in[0,1]
        . 
    \end{equation}
    Here we identify tangent and cotangent space of $\SLD$ through Killing form of $\sld$, which is non-degenerate since $\sld$ is simple (see Appendix A). 
    Recall also that Killing form $K$ on $\sld$ can be computed as
    \begin{equation*}
        K(M_1,M_2)
        =
        4
        \tr(M_1 M_2),
        \quad 
        M_1,M_2\in\sld.
    \end{equation*}
    From Pontryagin Maximum Principle and the theory of left-invariant Hamiltonian systems on Lie Groups (see Appendix A) 
    we obtain that if $\tu$ is a local maximum for $\Phi_T$, then there exists a Lipschitz function $\e : [0,T] \to \sld$ satisfying 
        \begin{equation}
		\label{adj_syst}
		\left\{
			\begin{array}{l}
	 		\dot{\e}(t)
		 		=
		 		\big[
                \e(t),
                \:
		 		\tu(t)A+\big(1-\tu(t)\big)B
		 		\big],
		 		\\
		 		\e(0)
		 		=
		 		\e_0.
			 	\end{array}
			 	\right.
			\end{equation} 
    for a.e. $t\in [0,T]$ (see Proposition \ref{LieGroup1} and Equation \ref{HamSystLieG2}). We will call this ODE adjoint system and its solution $\e$ adjoint trajectory. Moreover, the two following conditions must hold:
	\begin{align}
		\label{max}    
		H(\e(t),\tu(t))
		&= 
		\max_{u\in[0,1]} 
		H(\e(t),u),
		\quad 
		\text{for a.e. } t\in[0,T], 
		\\
		\label{transv}
		\e(T;\tu)
        &=
        X(T;\tu)-\frac{\tr(X(T;\tu))}{2}\I.
    \end{align}   
    The latter condition is known as \emph{transversal condition}, see Lemma \ref{lemma_transv}.
    \\
    We will call Pontryagin \emph{extremal} a couple $(X(\cdot,u),\e(\cdot;u))$ satisfying all necessary conditions prescribed by PMP. 
\end{section}

\begin{section}{General properties of extremals}
    In the following, if it is not important to specify the control we will denote with $X$ the solution of \eqref{Group_eq} instead of $X(\cdot;u)$. 
\\
    The solution $\e$ to equation \eqref{adj_syst} can be expressed as
    \begin{equation}
        \label{sol_eta}
        \e(t)
        =
        X(t)^{-1}
        \e_0
        X(t).
    \end{equation}
    Indeed, differentiating the previous expression and recalling that 
    \begin{equation*}
        \frac{d}{dt}[X(t)^{-1}]
        =
        -
        X(t)^{-1}
        \dot{X}(t)
        X(t)^{-1},
    \end{equation*}
    one obtains equation \eqref{adj_syst}.
    \\[5pt]
    Thus, from \eqref{sol_eta}, we can see that the determinant along any solution of \eqref{adj_syst} is constant.
    For $M\in\sld$, we have the formula $2\det(M)=-\tr(M^2)$. In what follows, it will be more convenient to use the expression $\tr(M^2)$ instead of $\det M$.
    \\
    Notice that from \eqref{transv}, it follows that
		$$
		\tr( \e(T)^2 )
            =
            \frac{1}{2}
            \big(
                \tr X(T)
            \big)^2
            -
            2
            .
		$$
    So, since 
    $
    X(T)
    \in
    \SLD
    $, 
    then 
    $
    \tr(X(T))^2 
    \geq 
    4,
    $
    at least for large $T$,
    and so 
    $
    \tr( \e(T)^2 )
    > 
    0
    $.
    Moreover, since $\tr(\e^2)$ is constant, we can deduce that $\tr(\e(t)^2)>0$ for all $t\in[0,T]$. 
    In particular, $\e(t)$ must have real eigenvalues for all $t\in[0,T]$.
    \\[5pt]
    \begin{subsection}{Invariant Hyperboloid}
    Equation \eqref{adj_syst} is invariant after a dilation, in the sense that for any $c\in\R$, if $\e(\cdot)$ is a solution with initial value $\e(0)=\e_0$, then $\tilde{\e}(t)=c\e(t)$ is a solution with initial value $\Tilde{\e}(0)=c\e_0$.
    \\
    So, it is not restrictive to suppose that $\tr(\e(t)^2) = \tr (\e_0 ^2) = 1$. Hence, the trajectory $\e$ is contained in the set  
    \begin{equation*}
        H
        :=
        \{
        M\in\sld 
        \:
        |
        \:
        \tr (M^2) = 1
        \},
    \end{equation*}
    which is a (connected) hyperboloid in the three-dimensional vector space $\sld$.
    To see this, just take 
    \begin{equation*}
        M
        =
        \begin{pmatrix}
            m_1 & m_2 \\
            m_3 & -m_1
        \end{pmatrix}.
    \end{equation*}
    Then 
    \begin{equation*}
        \tr(M^2)
        =
        2
        (
        m_1 ^2 
        + 
        m_2 m_3
        )
        =
        2
        \left(
        m_1 ^2 
        + 
        \left(
            \frac{m_2+m_3}{2}
        \right)^2
        -
        \left(
            \frac{m_2-m_3}{2}
        \right)^2
        \right).
    \end{equation*}
    Since $A,B,[A,B]$ are linearly independent, we can write 
    \begin{equation*}
        M
        =
        \al A
        +
        \be B
        +
        \gamma [A,B],
    \end{equation*}
    and the equation defining $H$ reads
    \begin{equation*}
        \al^2 \tr(A^2)
        +
        \be^2 \tr(B^2)
        +
        \gamma^2 \tr([A,B]^2)
        +
        2\al\be \tr(AB)
        =
        1.
    \end{equation*}
    To simplify the previous expression, we will use the following result.
    \begin{lem}
        \label{TrBracket}
        Let $A,B\in \sld$. Then
        \begin{equation}
            \label{traccia_bracket}
            \tr([A,B]^2)
            =
            2
            \tr(AB)^2
            -
            2
            \tr(A^2)\tr(B^2).
        \end{equation}
    \end{lem}
    \begin{proof}
        From the definition of bracket, we have
        \begin{equation*}
            \tr([A,B]^2)
            =
            2\tr((AB)^2)
            -
            2\tr(A^2B^2).
        \end{equation*}
        Since $A,B\in \sld$, $A^2$ and $B^2$ are scalar, that is $A^2=\la^2\I$ and $B^2=\mu^2\I$. So
        \begin{equation*}
            \tr(A^2B^2)
            =
            2\la^2\mu^2
            =
            \frac{1}{2}
            \tr(A^2)
            \tr(B^2).
        \end{equation*}
        On the other hand, the minimal polynomial of $AB$ is 
        \begin{equation*}
            (AB)^2
            -
            \tr(AB)AB
            +
            \det(AB)\I
            =
            0,
        \end{equation*}
        and computing the trace of the expression above leads to
        \begin{equation*}
            \tr((AB)^2)
            =
            \tr(AB)^2
            -
            2\det A \det B
            =
            \tr(AB)^2
            -
            \frac{1}{2}
            \tr(A^2)
            \tr(B^2).
        \end{equation*}
        and so
        \begin{equation*}
            \tr([A,B]^2)
            =
            2\tr((AB)^2)
            -
            2\tr(A^2B^2)
            =
            2
            \tr(AB)^2
            -
            2
            \tr(A^2)\tr(B^2).
        \end{equation*}
        as we wanted to prove.
    \end{proof}
    Using the previous Lemma, we can rewrite the equation for $H$ as
    \begin{equation*}
        \al^2 \tr(A^2)
        +
        \be^2 \tr(B^2)
        +
        2\gamma^2 \big(
        \tr(AB)^2 
        - 
        \tr(A^2)\tr(B^2)
        \big)
        +
        2\al\be \tr(AB)
        =
        1.
    \end{equation*}
    \end{subsection}
    \begin{subsection}{Switching function, \emph{bang-bang} and singular extremals}
        Equation \eqref{max} can be rewritten more explicitly as
        \begin{align}
            \tr(
            \e(t) B)
            +
            \tu(t)
            \tr(
            \e(t) (A-B)
            )
            &=
            \max _{u\in[0,1]}
            \big[
            \tr(
            \e(t) B)
            +
            u
            \tr(
            \e(t) (A-B)
            )
            \big] 
            \\
            &=
            \tr(
            \e(t) B)
            +
            \max _{u\in[0,1]}
            u
            \tr(
            \e(t) (A-B)
            )
            ,
        \end{align}
        for almost every $t\in[0,T]$.
        Hence, we are lead to define $\varphi:[0,+\infty)\to\R$,
        \begin{equation}
        \label{swFunction}
            \varphi(t)
            =
            \tr(
            \e(t) (A-B)
            ).
        \end{equation}
        We will call $\varphi$ \emph{switching function}. Notice that $\varphi$ is at least Lipschitz.
        If $\varphi$ has only isolated zeros, then a control $u$, in order to satisfy Pontryagin Maximum Principle, must be of the form
        \begin{equation}
            \label{contr_max}
            u(t)
            =
            \left\{
            \begin{array}{l}
                1 \quad \text{ if } \varphi(t) > 0, \\
                0 \quad \text{ if } \varphi(t) < 0,
            \end{array}
            \right.
        \end{equation}
        that is, the isolated zeros of $\varphi$ are exactly the times when there are a \emph{switch} in the dynamics. 
        \\
        If $\varphi(t_0)=0$, we will call $t_0$ a \emph{switching time} and the corresponding $\e(t_0)$ \emph{switching point}.
        \\
        If $\varphi$ has any nonisolated zero, then Pontryagin Maximum Principle does not determine the control directly, but in our case it is still possible to bypass this obstacle and determine the control. 
        \begin{defn}[\emph{bang-bang} control, Singular control]
        \label{RegSingExtr}
            We will call \emph{bang-bang} a control $u$ and its corresponding extremal if $\varphi$ has only isolated zeros. 
            \\
            If instead  $\varphi\equiv 0$, then we will call the associated control $u$ and its extremal \emph{singular}.
        \end{defn}
        A control $u$ which takes only extremal values, i.e. $u(t)\in\{0,1\}$ for a.e. $t\in[0,T]$, is called \emph{bang-bang}.
        We will first focus on \emph{bang-bang} extremals and then we will discuss separately the singular case.
    \end{subsection}
    
\end{section}

\begin{section}{\emph{Bang-bang} extremals}
    \begin{subsection}{Properties of \emph{bang-bang} extremals}
        In order to determine the structure of Pontryagin \emph{bang-bang} extremals, it will be useful to study the plane of switching points:
        \begin{equation}
            \label{sw_curve}
            \Pi
            =
            \{
            M\in\sld
            \;
            |
            \; 
            \tr(MA)
            =
            \tr(MB)
            \}.
        \end{equation}
        This is the set of matrices in $\sld$ for which the function $\varphi$ is zero.
        \\
        Since the matrices $A,B,[A,B]$ form a bases of $\sld$ by hypothesis, we can write the equation $\tr(MA)=\tr(MB)$ as 
        \begin{equation*}
            \al \tr(A^2) + \be \tr (AB) + \gamma \tr([A,B]A)
            =
            \al \tr(AB) + \be \tr (B^2) + \gamma \tr([A,B]B),
        \end{equation*}
        that is 
        \begin{equation}
        \label{swCoord}
            \be
            =
            \frac{
            \tr(AB)-\tr(A^2)
            }{
            \tr(AB)-\tr(B^2)
            }\al. 
        \end{equation}
        So, if we impose $\tr(M^2)=1$, that is if we intersect the hyperboloid $H$ with the plane $\Pi$ we obtain the equation of a conic
        \begin{align}
        \label{EqSwCurve}
            \frac{
                \big(
                    2\tr(AB)-\tr(A^2)-\tr(B^2)
                \big)
                \tr
                \big(
                    [A,B]^2
                \big)
            }{
            \big(
                \tr(AB)-\tr(B^2)
            \big)^2
            }
            \al^2
            + 
            \tr
                \big(
                    [A,B]^2
                \big)
            \gamma^2
            =
            1.
        \end{align}
        \begin{figure}[ht]
            \centering
            \includegraphics[width=0.49\linewidth]{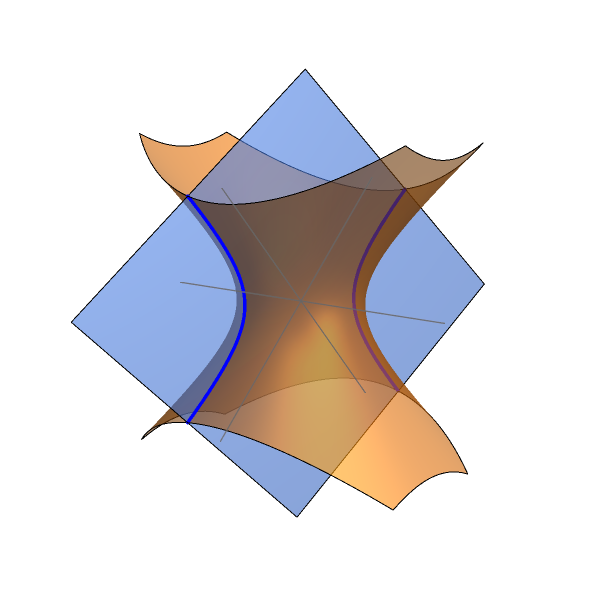}
            \includegraphics[width=0.49\linewidth]{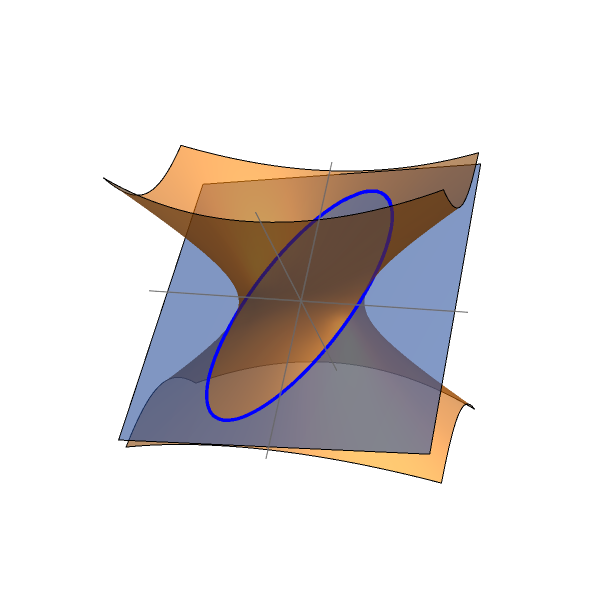}
            \caption{Examples of switching curves. In both pictures you can see the orange hyperboloid $H$, the blue plane $\Pi$ and their intersection, the switching curve $S$.}
            \label{figSwCurve}
        \end{figure}
        \begin{defn}[Switching curve]
            We will call \emph{switching curve} $S$ the intersection curve $\Pi\cap H$, whose equation is \eqref{EqSwCurve}. 
        \end{defn}
        So, the coefficient in front of $\al$ is positive if and only if 
        \begin{equation*}
            \big(
                2\tr(AB)-\tr(A^2)-\tr(B^2)
            \big)
            \tr
            \big(
            [A,B]^2
            \big)
            >
            0.
        \end{equation*}
        Hence, if 
        $
            \tr
            \big(
            [A,B]^2
            \big)
            >0
        $,
        then the switching curve is an ellipse if
        $
        \tr(AB)
        >
        \frac{
        \tr(A^2)+\tr(B^2)
        }{2}
        $, 
        and is an hyperbola if 
        $
            \tr
            \big(
            [A,B]^2
            \big)
            >0
        $
        and 
        $
        \tr(AB)
        <
        \frac{
        \tr(A^2)+\tr(B^2)
        }{2}
        $.
        \\
        If instead 
        $
            \tr
            \big(
            [A,B]^2
            \big)
            <
            0
        $
        ,
        the switching curve is an hyperbola if 
        $
        \tr(AB)
        <
        \frac{
        \tr(A^2)+\tr(B^2)
        }{2}
        $.
        Notice that 
        $
        \tr(AB)
        >
        \frac{
        \tr(A^2)+\tr(B^2)
        }{2}
        $
        implies 
        $
            \tr
            \big(
            [A,B]^2
            \big)
            >0
        $,
        so we can neglet the case
        $
        \tr(AB)
        >
        \frac{
        \tr(A^2)+\tr(B^2)
        }{2}
        $
        and
        $
            \tr
            \big(
            [A,B]^2
            \big)
            <0
        $
        .
    \end{subsection}
    \begin{subsection}{Hyperboloid foliation}
        Let us now fix an initial point on the hyperboloid:
        \begin{equation*}
            \e_0\in H.
        \end{equation*}
        Define the half spaces
        \begin{align*}
            \Pi^+
            &=
            \{
            M\in\sld
            \;
            |
            \;
            \tr(MA)>\tr(MB)
            \},
            \\
            \Pi^-
            &=
            \{
            M\in\sld
            \;
            |
            \;
            \tr(MA)<\tr(MB)
            \},
        \end{align*}
        and the half hyperboloids
        \begin{align*}
            H^+
            &=
            H
            \cap
            \Pi^+,
            \\
            H^-
            &=
            H
            \cap
            \Pi^-
            .        
        \end{align*}
        From condition \eqref{contr_max}, we can see that if $\e_0\in H^+$, the control $u$, in order to satisfy Pontryagin Maximum Principle, must be equal to $1$ until the trajectory $\e$ meets the switching curve $S$.
        Notice that, as far as $u(t)=1$, the Hamiltonian function of the problem is $H(\e(t),1)=\tr(\e(t)A)$, hence the quantity $\tr(\e(t)A)$ is constant. 
        \\
        Similarly, if $\e(t)\in H^-$ for $t\in[\tau_1,\tau_2]$, for some $0<\tau_1<\tau_2$, then $H(\e(t),0)=\tr(\e(t)B)$ is constant. 
        \\[5pt]
        Define the planes
        \begin{align*}
            \Pi_{A,c}
            &=
            \{
            M\in\sld 
            \;
            |
            \;
            \tr(MA)=c
            \}
            \\
            \Pi_{B,c}
            &=
            \{
            M\in\sld 
            \;
            |
            \;
            \tr(MB)=c
            \},
        \end{align*}
        for $c\in\R$.
        \\
        So, if $\e_0\in H^+$ is such that $\tr(\e_0 A)=c_0$, then the trajectory $\e$ is contained in the curve $H^+\cap \Pi_{A,c_0}$ (see Figure \ref{CasiIntersez}). 
        In particular, we can suppose that geometrically the trajectory $\e$ coincides with the whole curve $H^+\cap \Pi_{A,c_0}$. 
        Indeed, if $\e$ accumulates at some point, then the control $u(t)$ is identically equal to $1$ or to $0$ for every $t$ sufficiently large and, as we will see, the corresponding Lyapunov exponents are easy to compute.
        \\
        Now, we want to study all the possible configuration of $H,S$ and the planes $\Pi_{A,c_0},\Pi_{B,c_0}$. Consider the straight line given by the intersection of the three planes $\Pi,\Pi_{A,c_0},\Pi_{B,c_0}$:
        \begin{equation*}
            r_{c_0} : 
            \left\{
            \begin{array}{l}
                 \tr(MA)=c_0,\\
                 \tr(MB)=c_0,\\
                 \tr(MA)=\tr(MB).
            \end{array}
            \right.
        \end{equation*}
        There are three cases: 
        either this line does not intersect the hyperboloid $H$, 
        either it is tangent to the hyperboloid, 
        or it has two distinct intersection with the hyperboloid. 
        \\
        We will see later that if $r_{c_0}$ is tangent, then the corresponding trajectory $\e$ is constant and in particular it is a singular trajectory. So, we will consider first the non-tangent cases and we postpone the tangent case to the Section of singular extremals.
        \\[5pt]
        If $r_{c_0}$ does not intersect the hyperboloid, then the trajectory $\e$ starting from $\e_0$ does not meet the switching curve $S$. This means that $\e$ is confined either in the region $H^+$ or $H^-$, so the control $u$ is identically equal to 1 or $0$ for all times.
        \\[5pt]
        If instead there are two distinct intersection between $r_{c_0}$ and $H$, then there are three possibilities (see Figure \ref{CasiIntersez}):
        \begin{enumerate}
            \item there is exactly one switch;
            \item there are two switches and after the second switch the adjoint trajectory does not meet anymore the switching curve;
            \item the adjoint trajectory is periodic with infinite switches.
        \end{enumerate}
        Since we are interested in computing the Lyapunov exponents of the system, the case with no switches and the cases with finite number of switches are the same, because the resulting trajectories are asymptotically equal.  
        \\
        So, to resume, we have obtained the following result about the structure of \emph{bang-bang} extremals.
        \begin{thm}
            Any admissible control corresponding to a \emph{bang-bang} extremal must be in one of the following two forms:
            \begin{itemize}
                \item constant controls, that is either $u(\tau)\equiv 1$ or $u(\tau)\equiv 0$ for all $\tau \in [0,+\infty);$
                \item \emph{bang-bang} periodic controls, that is controls in the form 
                \begin{equation}
                    \label{BangBangPer}
                    u(\tau)
                    =
                        \begin{cases}
                            1 \quad \text{ if } \tau \in [0,t), \\
                            0 \quad \text{ if } \tau \in [t,T),
                        \end{cases}
                \end{equation}
                for some $T>t>0$ and then extended by periodicity on the whole interval $[0,+\infty)$.
            \end{itemize}
        \end{thm}
        Another important observation is the following.
        \begin{prop}
            \label{TransvImplPeriodic}
            A couple $(X,\e)$ solution of \eqref{Group_eq} and \eqref{adj_syst} respectively satisfies transversal condition \eqref{transv} if and only if $\e$ is periodic of period $T$. 
        \end{prop} 
        \begin{proof}
            If \eqref{transv} is satisfied, then 
            \begin{equation*}
                X(T)^{-1}
                \e(0)
                X(T)
                =
                \e(T)
                =
                X(T)
                -
                \frac{\tr(X(T))}{2}\I,
            \end{equation*}
            and simply multiplying both sides by $X(T)$ on the left and by $X(T)^{-1}$ on the left one obtains that also 
            \begin{equation*}
                \e(0)
                =
                X(T)
                -
                \frac{\tr(X(T))}{2}\I
                =
                \e(T),
            \end{equation*}
            that is, $\e$ is periodic of period $T$.
            \\
            Viceversa,
            if $\e(0)=\e(T)$, then, from $\e(T)=X(T)^{-1}\e(0)X(T)$, it follows that 
            \begin{equation*}
                X(T)\e(0)
                =
                \e(0)X(T),
            \end{equation*}
            that is, $\e(0),\e(T)$ and $X(T)$ commutes. But then also $X(T)-\frac{\tr(X(T))}{2}\I$ and $\e(T)$ commutes, and since they are both matrices in $\sld$, which is semisimple, they must be proportional, i.e. there is $c\in\R$ such that
            :
            \begin{equation*}
                c\e(T)
                =
                X(T)
                -
                \frac{\tr(X(T))}{2}\I.
            \end{equation*}
            Hence transversal condition \eqref{transv} is satisfied.
        \end{proof}
        \begin{figure}[!ht]
            \centering
            \includegraphics[width=0.48\linewidth]{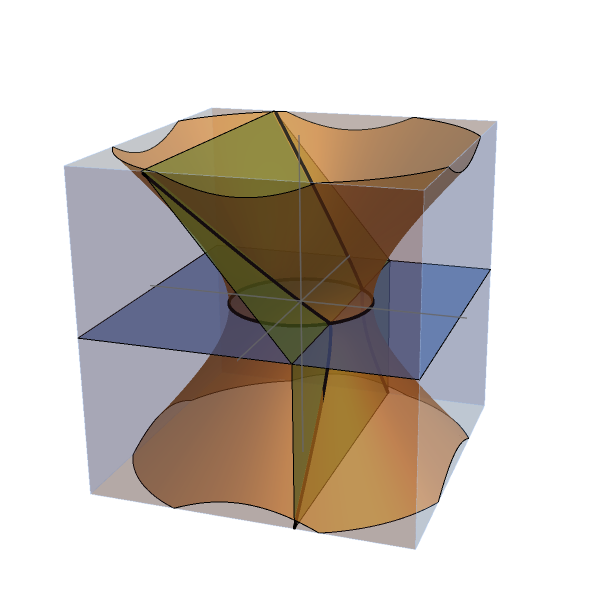}
            \includegraphics[width=0.48\linewidth]{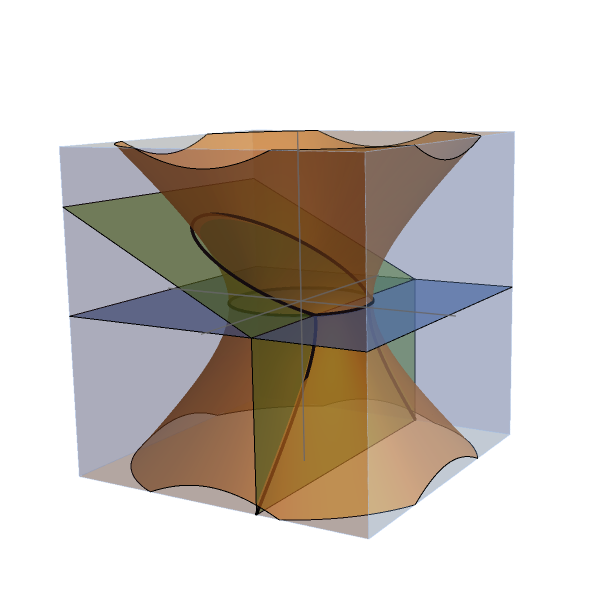}
            \includegraphics[width=0.48\linewidth]{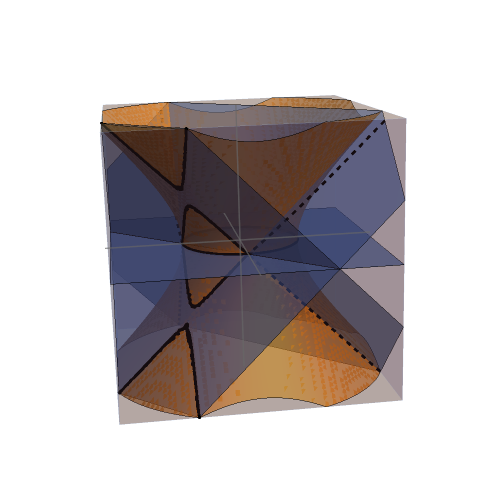}
            \includegraphics[width=0.48\linewidth]{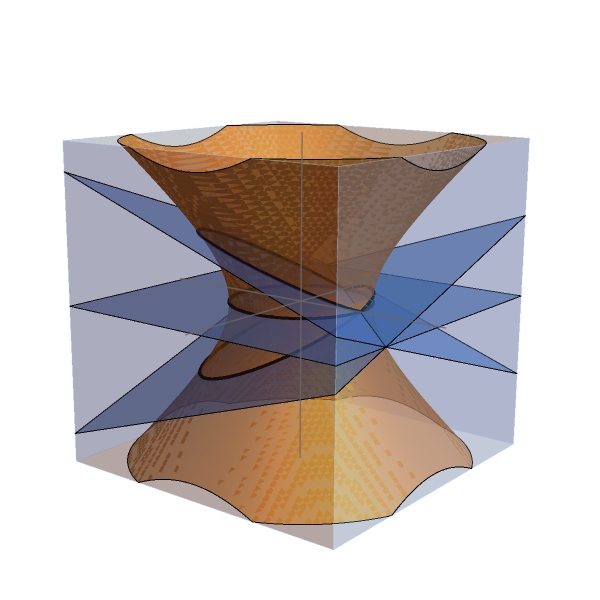}
            \caption{The figure illustrates the possible cases when the switching curve is a horizontal circle: 
            in the first figure the intersection between the Hyperboloid $H$ and the planes $\Pi_{A,c_0}$ and $\Pi_{B,c_0}$ are two hyperbolas and the corresponding adjoint trajectory has only one switch. 
            In the second figure the intersection with one plane is an ellipse and the other is a hyperbola, and the resulting adjoint trajectory has two switch. 
            In the third case, the two intersections between the hyperboloid and the planes are hyperbolas, but in this case there are three connected components and one of them is closed, which corresponds to a periodic adjoint trajectory. 
            The last figure is similar to the third one but with ellipses instead of hyperbolas.}
            \label{CasiIntersez}
        \end{figure}
        \noindent
        Now, we want to find a precise expression for the intersection points between the line $r_{c_0}$ and the switching curve $S$.
        \\
        It is convenient to parameterize such initial data with the corresponding value $c_0$. More precisely, we are looking for initial data $\e_0\in\sld$ such that
        \begin{equation}
            \label{2inters}
            \left\{
            \begin{array}l
                 \tr(\e_0^2)=1,  \\[3pt]
                 \tr(\e_0A)=\tr(\e_0B), \\[3pt]
                 \tr(\e_0A)=c_0.
            \end{array}
            \right.
        \end{equation}
         Using the same notation introduced above, we can write 
         $
         \e_0
         =
         \al A
         +
         \be B
         +
         \gamma [A,B],
         $
         and one can find the corresponding solutions for $\al,\be,\gamma$:
         \begin{align}
         \nonumber
             \al
             &=
             \frac{
                 2
                 \big(
                    \tr(B^2)-\tr(AB)
                 \big)
             }{
                \tr\big([A,B]^2\big)
             }
             c_0,
             \\[3pt]
            \label{CoeffEta}
             \be
             &=
             \frac{
                2
                \big(
                    \tr(A^2)-\tr(AB)
                \big)
             }{
                \tr\big([A,B]^2\big)
             }
             c_0,
             \\[3pt]
             \nonumber
             \gamma
             &=
             \pm 
             \frac{
                \sqrt{
                        2
                        \big(
                            \tr(A^2) + \tr(B^2) -2\tr(AB)
                        \big)
                        c_0 ^2
                        +
                        \tr\big([A,B]^2\big)
                }
             }{
                |\tr\big([A,B]^2\big)|
             }.
         \end{align}
         Thus, there are two distinct intersection if $c_0$ is such that 
            \begin{equation*}
                2
                \big(
                    \tr(A^2) + \tr(B^2) - 2\tr(AB)
                \big)
                c_0 ^2
                +
                \tr([A,B]^2)
                >
                0.
            \end{equation*}
        If  
        $
        \tr(A^2) + \tr(B^2) -2\tr(AB) 
        >
        0
        $
        and 
        $\tr([A,B]^2)>0$ (the switching curve is an hyperbola), then the previous condition holds for every $c_0\in\R$.
        If instead 
        $
        \tr(A^2) + \tr(B^2) -2\tr(AB) 
        >
        0
        $
        and 
        $\tr([A,B]^2)<0$ (the switching curve is again an hyperbola)
        this condition corresponds to
        \begin{equation*}
            c_0
            >
            \frac{1}{\sqrt{2}}
            \sqrt{
                \frac{
                    \tr\big([A,B]^2\big)
                }{
                    2\tr(AB) - \tr(A^2) - \tr(B^2)
                }
            }
            \quad 
            \text{ or }
            \quad 
            c_0
            <
            -
            \frac{1}{\sqrt{2}}
            \sqrt{
                \frac{
                    \tr\big([A,B]^2\big)
                }{
                    2\tr(AB) - \tr(A^2) - \tr(B^2)
                }
                }.
        \end{equation*}
        If instead   
        $
        \tr(A^2) + \tr(B^2) -2\tr(AB) 
        <
        0
        $
        ,
        this implies $\tr\big([A,B]^2\big)>0$ (hence the switching curve is an ellipse)
        \footnote{It a consequence of \ref{TrBracket} and $2c>a+b \implies c^2>ab$, which follows from $ab<\frac{1}{4}(a+b)^2$.}
        ,
        hence $c_0$ can assume the values
        \begin{equation*}
            -\frac{1}{\sqrt{2}}
            \sqrt{
                \frac{
                    \tr\big([A,B]^2\big)
                }{
                    2\tr(AB) - \tr(A^2) - \tr(B^2)
                }
            }
            <
            c_0
            <
            \frac{1}{\sqrt{2}}
            \sqrt{
                \frac{
                    \tr\big([A,B]^2\big)
                }{
                    2\tr(AB) - \tr(A^2) - \tr(B^2)
                }
                }.
        \end{equation*}
        \\[5pt]
        Notice that if $\gamma=0$, then the line $r_{c_0}$ is tangent to the hyperboloid.
        \\[5pt]
        As already mentioned in the Introduction, if the control $u$ in equation \eqref{Group_eq} is constant equal to $1$, then the corresponding principal Lyapunov exponent is equal to the  positive eigenvalue of $A$, i.e.
        \begin{equation*}
            \ell_{u=1}(A,B)
            =
            \sqrt{
                \frac{\tr(A^2)}{2}
            }.
        \end{equation*}
        Analogously, if the control $u$ is constant equal to $0$, then
        \begin{equation*}
            \ell_{u=0}(A,B)
            =
            \sqrt{
                \frac{\tr(B^2)}{2}
            }.
        \end{equation*}
        Clearly, since we are assuming $\tr(A^2)\geq \tr(B^2)$, we have $\ell_{u=1} \geq \ell_{u=0}$.
        \\
        So, it remains to analyze the case of periodic controls.
    \end{subsection}

    \begin{subsection}{\emph{Bang-bang} periodic controls}
    So far, we proved that there are two qualitatively different type of adjoint trajectories: with a finite number of switches or periodic trajectories with an infinite number of switches. 
    \par 
    If the adjoint trajectory is periodic, then also the switching function \eqref{swFunction} is periodic and then the associated control is periodic too. 
    However, not every \emph{bang-bang} periodic control as in \eqref{BangBangPer} satisfy PMP. As we will see, for any $T>0$, the control in \eqref{BangBangPer} satisfies PMP only for a specific value of the switching time $t\in[0,T]$. 
    \\
    In this Section we are going to describe how to find \emph{all} periodic solution to equation adjoint equation \eqref{adj_syst}, and so we will be able to characterize all periodic controls satisfying Pontryagin Maximum Principle.
    \\[5pt]
    From \eqref{sol_eta}, finding all periodic solution to equation adjoint equation \eqref{adj_syst} amounts to find $X(T)$ in the form
        \begin{equation*}
            X(T)
            =
            e^{t A}
            e^{s B},
        \end{equation*}
    with $t+s=T$ and $0,t,T$ switching times. 
    \par
    We will proceed as follows: for every $t$ we will find $s=s(t)$ such that $0,t,T=t+s(t)$ are switching points.
    The following result simplifies this task.
    \begin{prop}
        Define the \textit{bang-bang} control
        \begin{equation*}
                u(\tau)
                =
                \begin{cases}
                    1 \quad \text{if} \quad \tau\in[0,t], \\
                    0 \quad \text{if} \quad \tau\in[t,T].
                \end{cases}
        \end{equation*}
        Let $X$ be the corresponding solution of \eqref{Group_eq} and $\e$ be the solution of \eqref{adj_syst}. 
        Suppose that $\e$ satisfies the condition of transversality \eqref{transv}. 
        If $t$ is a switching point, then also $0$ and $T$ are switching points.
        \end{prop}
        \begin{proof}
            The solution of equation \eqref{adj_syst} is 
            \begin{equation*}
                \e(\tau)
                =
                X(\tau)^{-1}
                \e_0
                X(\tau).
            \end{equation*}
            Since $\e$ must satisfy the transversal condition, we know by Proposition \ref{TransvImplPeriodic} that $\e$ is periodic of period $T$, so it must be
            \begin{equation*}
                \e_0
                =
                X(T)
                -
                \frac{
                    \tr
                    \big(
                        X(T)
                    \big)
                }{
                    2
                }
                \I.
            \end{equation*}
            So, putting together the last two equations with the formula for $X$, we obtain that 
            \begin{equation*}
                \e(\tau)
                =
                \left\{
                \begin{array}{l}
                     e^{(t -\tau)A}
                     e^{s B}
                     e^{\tau A}
                     -
                     \frac{
                        \tr
                        \left(
                            X(T)
                        \right)
                     }{
                        2
                     }
                     \I
                     \quad 
                     \text{ if }
                     \tau\in[0,t]
                     \\[4pt]
                     e^{(s -\tau+t)B}
                     e^{t A}
                     e^{\tau-t B}
                     -
                     \frac{
                        \tr
                        \left(
                            X(T)
                        \right)
                     }{
                        2
                     }
                     \I
                     \quad 
                     \text{ if }
                     \tau\in[t,t+s]
                \end{array}
                \right.
            \end{equation*}
            We are assuming that $t$ to be a switching time, so it satisfy 
            \begin{equation*}
                \tr(
                \e(t) A
                )
                =
                \tr(
                \e(t) B
                ).
            \end{equation*}
            Using the formula for $\e$ that we have just obtained, we get
            \begin{equation}
                \label{eeqq}
                \tr(
                e^{s B}
                e^{t A}
                A
                )
                =
                \tr(
                e^{s B}
                e^{t A}
                B
                ).
            \end{equation}
            The left hand side can be rewritten as 
            \begin{equation*}
                \tr(
                e^{s B}
                e^{t A}
                A
                )
                =
                \tr(
                e^{s B}
                A
                e^{t A}
                )
                =
                \tr(
                e^{t A}
                e^{s B}
                A
                )
                =
                \tr(
                X(T)
                A
                ).
            \end{equation*}
            Similarly, the right hand side of \eqref{eeqq} is equal to 
            $
            \tr(
            X(T)B
            )
            $. 
            Hence, it follows that also 
            \begin{equation}
            \label{SwEqFinTime}
                \tr(
                X(T)A
                )
                =
                \tr(
                X(T)B
                ),
            \end{equation}
            which implies that $T$ (and thus, since $\e$ is $T$-periodic, also $0$) is a switching time.
        \end{proof}
        \begin{rem}
            If we consider a solution of \eqref{Group_eq} with one switch, the final point is in the form 
            \begin{equation*}
                X(T)
                =
                X(t,s)
                =
                e^{t A}
                e^{s B}
            \end{equation*}
            where $T=t+s$. So, we can see the cost functional as a function of $t$ and $s$:
            \begin{equation*}
                \Phi(t,s)
                =
                \tr(
                e^{t A}
                e^{s B}
                ).
            \end{equation*}
            With this notation, if $t$ is a switching point, from the previous proof follows that $X(T)$ must satisfy Equation \eqref{SwEqFinTime}, which can be rewritten as 
            \begin{equation}
                \label{swTime}
                \frac{
                    \pa \Phi
                }{
                    \pa t
                }(t,s)
                =
                \frac{
                    \pa \Phi
                }{
                    \pa s
                }(t,s).
            \end{equation}
        \end{rem}
        To summarize, given any $t>0$, equation \eqref{swTime} give us a condition to find $s>0$ as a function of $t$ such that the control
        \begin{equation*}
                u(\tau)
                =
                \begin{cases}
                    1 \quad \text{if} \quad \tau\in[0,t], \\
                    0 \quad \text{if} \quad \tau\in[t,t+s],
                \end{cases}
        \end{equation*}
        satisfies Pontryagin Maximum Principle on the time interval $[0,t+s]$. The corresponding solution to the adjoint system $\e$ satisfies transversal condition on the same time interval, so one can extend this control by periodicity on the whole line $[0,+\infty)$.  
        \par
        Concerning the value of the principal Lyapunov exponent, when the adjoint trajectory is periodic it can be computed as follows:
        \begin{align*}
            \ell_{\text{per}}(T)
            &=
            \limsup_{n\to+\infty}
            \frac{1}{n T} 
            \log
            \|X(nT)\|
            =
            \\
            &=
            \limsup_{n\to+\infty}
            \frac{1}{n T} 
            \log
            \|X(T)^n\|
            =
            \\
            \label{perLyap}
            &=
            \limsup_{n\to+\infty}
            \frac{1}{n T} 
            \log
            \la(T)^n
            =
            \frac{1}{T} 
            \log
            \la(T),
            \numberthis
        \end{align*}
        Where $\lambda(T)$ is the greatest eigenvalue of $X(T)$.
        We can further simplify the expression for $\ell_{\text{per}}(T)$. Indeed, the functional of our Optimal Control Problem does not appear explicitly in the formula for $\ell_{\text{per}}(T)$. However, we can deduce the value the eigenvalue from the trace solving 
        \begin{equation*}
            \la(T)^2
            -
            \tr(X(T))
            \la(T) 
            +
            1
            =
            0,
        \end{equation*}
        which gives
        \begin{equation*}
            \la(T)
            =
            \frac{
                \tr(X(T))
                +
                \sqrt{
                    \tr(X(T))^2-4
                }
            }{2}.
        \end{equation*}
        Using the identity 
        $
        \arcosh x
        =
        \log(
            x
            +
            \sqrt{x^2-1}
        )
        $, we obtain
        \begin{equation*}
        \label{1SwLyap}
            \ell_{\text{per}}(T)
            =
            \frac{
                \arcosh
                \left(
                \frac{ 
                \tr(X(T))
                }{2}
                \right)
            }{T}.
        \end{equation*}
\end{subsection}
    
    \begin{subsection}{Principal Lyapunov exponent for periodic controls}
    Now, we are now going to compute the upper bound for Principal Lyapunov exponent of periodic controls. 
    As we saw previously, in order to find the extremals of our problem we have to study the functional 
$
\Phi:
[0,+\infty)
\times
[0,+\infty)
\to 
\R
$
\begin{equation*}
\label{Funct}
\Phi(t,s)
=
\tr(
e^{t A}
e^{s B}
).
\end{equation*}
Since equations \eqref{Group_eq} and \eqref{adj_syst} are invariant after a change of basis, if $A$ and $B$ are not nilpotent (see Remark \ref{nilp} for the nilpotent case), without loss of generality we can assume
\begin{equation*}
    A
    =
    \begin{pmatrix}
        \la & 0 \\
        0 &-\la 
    \end{pmatrix}
    ,
    \quad
    B
    =
    P
    \begin{pmatrix}
     \mu & 0 \\
     0 & -\mu
    \end{pmatrix}
    P^{-1},
    \quad 
    P
    =
    \begin{pmatrix}
        a & b \\
        c & d
    \end{pmatrix},
\end{equation*}
with 
$
\la,\mu
\in 
\R 
\cup 
i\R
$
and $P\in\mathrm{SL}_2(\C)$. 
Up to exchange the role of $A$ and $B$, we can assume $\tr(A^2)\geq\tr(B^2)$.
\begin{prop}
    With this notations, the functional in \eqref{Funct} is
    \begin{equation}
        \label{TrEndpoint}
        \Phi(t,s)
        =
        2
        \cosh(\la t)
        \cosh(\mu s)
        +
        \frac{\tr(AB)}{\lambda \mu}        
        \sinh(\la t)
        \sinh(\mu s).
    \end{equation}
\end{prop}
This formula is a simple consequence of the following Lemma.
\begin{lem}
    If $M\in\sld$, with positive eigenvalue $\al > 0$, then 
    \begin{equation}
        \label{ExpMatr}
        e^{tM}
        =
        \cosh(\al t)
        \I
        +
        \frac{\sinh(\al t)}{\al}
        M.
    \end{equation}
\end{lem}
\begin{proof}
    From the characteristic polynomial of $M$, we know that 
    $
    M^2
    =
    \al^2 \I
    $
    , 
    and so 
    $
    M^{2k}
    =
    \al^{2k} \I
    $
    ,
    $
    M^{2k+1}
    =
    \al^{2k} M
    $
    .
    Using these facts in the series expansion of $e^{tM}$ one obtains the formula above. 
\end{proof}
\begin{rem}
    If $M\in\sld$ is nilpotent, i.e. $M^2=0$, then 
    \begin{equation}
        \label{ExpNilp}
        e^{tM}=\I+t M.
    \end{equation}
    So, if either $A$ or $B$ is nilpotent, one can just replace formula \eqref{TrEndpoint} with a formula for $\Phi$ obtained using \eqref{ExpNilp}.
\end{rem}
    \begin{rem}
    \label{nilp}
    From \eqref{TrEndpoint} we see immediately that $\Phi(T,0)$ coincide with the case of a trajectory without switch and 
    \begin{equation*}
        \frac{
            \pa \Phi
        }{
            \pa s        
        }
        (t,0)
        =
        \frac{\tr(AB)}{\la}
        \sinh (\la t).
    \end{equation*}
    On the other hand, if we define $\Psi(t,s)=2 \cosh(\la (t+s))$, which represent the cost of a trajectory without switch on the same interval of time, we obtain 
    \begin{equation*}
        \frac{
            \pa \Psi
        }{
            \pa s        
        }
        (t,0)
        =
        2\la
        \sinh (\la t).
    \end{equation*}
    Thus, if 
    $
     \frac{\tr(AB)}{\la}
    >
    2\la
    $,
    which means 
    \begin{equation*}
        \tr(AB)
        >
        \tr(A^2),
    \end{equation*}
    then the trace of a trajectory with a switch grows faster, at least for small times, than the trajectory without switch.
    \end{rem}
    \begin{subsubsection}{Imaginary eigenvalues}
        We are now going to examine a special case. Suppose that $\tr(A^2),\tr(B^2)<0$, i.e. $A,B$ have imaginary eigenvalues.
    \begin{prop}
    If both $A,B$ have imaginary eigenvalues, the functional $\Phi$ is 
    \begin{equation*}
        \Phi(t,s)
        =
        \tr
        \big(
        \exp(t A)
        \exp(s B)
        \big)
        =
        2
        \cos(\la t)
        \cos(\mu s)
        -
        2\gamma
        \sin(\la t)
        \sin(\mu s),
    \end{equation*}   
       Where 
        $
        \gamma
        =
        \frac{
            |\tr(AB)|
        }{
            \sqrt{
                \tr(A^2)\tr(B^2)
            }
        }
        $. 
    \end{prop}
    This is simply a consequence of formula \eqref{TrEndpoint} and $\cosh(i\la t)=\cos(\la t)$, $\sinh(i\mu s)=i\sin(\mu s)$.
    With this formula, the switching equation \eqref{swTime} reads
    \begin{equation*}
        (\gamma\la-\mu)
        \cos(\la t)
        \sin(\mu s)
        =
        (\gamma\mu-\la)
        \cos(\mu s)
        \sin(\la t),
    \end{equation*}
    from which we obtain immediately that there are always solutions for 
    $
    \cos(\la t),
    \cos(\mu s)
    =
    0,
    $
    that is
    \begin{equation*}
            (t,s)
		\in
		\left\{
			    \left(
                    \frac{\pi}{2\la}
					,
		          \frac{3\pi}{2\mu} 
			    \right)
                ,
                \left(
                    \frac{\pi}{2\la}
					,
		          \frac{\pi}{2\mu} 
			    \right)
			    ,
		    \left(
	   	        \frac{3\pi}{2\la}
		            ,
		          \frac{\pi}{2\mu} 
		    \right)
		\right\}.
    \end{equation*}
    From these choices the value of this functional is $\Phi(t,s)=\pm 2\gamma$.
    Using formula \eqref{1SwLyap}
    one obtains
    \begin{equation*}
        \ell_{\text{per,im}}(A,B)
        =
        \frac{1}{t+s}
        \arcosh
        (\gamma).
    \end{equation*}
    If $\tr(AB)<0$, then we have
    \begin{equation}
        \ell_{\text{per,im}}(A,B)
        =
        \frac{2}{
        \pi
        \left(
            \sqrt{\frac{-2}{\tr(A^2)}}
            +
            3\sqrt{\frac{-2}{\tr(B^2)}}
        \right)
        }
        \arcosh
        \left(
            \frac{
            -\tr(AB)
        }{
            \sqrt{
                \tr(A^2)\tr(B^2)
            }
        }
        \right).
    \end{equation}
    Notice that we have chosen the couple $(t,s)$ so to maximize $\frac{1}{t+s}$ under the hypothesis $\tr(B^2)<\tr(A^2)<0$.
    \\
    If instead $\tr(AB)>0$, then
    \begin{equation}
        \ell_{\text{per,im}}(A,B)
        =
        \frac{2}{
        \pi
        \left(
            \sqrt{\frac{-2}{\tr(A^2)}}
            +
            \sqrt{\frac{-2}{\tr(B^2)}}
        \right)
        }
        \arcosh
        \left(
            \frac{
            \tr(AB)
        }{
            \sqrt{
                \tr(A^2)\tr(B^2)
            }
        }
        \right).
    \end{equation}
    We will see in the next subsection how to deal with the extremals with 
    $
    \cos(\la t),
    \cos(\mu s)
    \neq
    0
    $.
    \end{subsubsection}
    \begin{subsubsection}{General case}
We suppose now 
$
\cosh(\la t),
\cosh(\mu s)
\neq 
0
$ 
(we already saw the case
$
\cosh(\la t),
\cosh(\mu s)
= 
0
$, 
which can happen if $\la,\mu\in i\R$).
With this assumption,
equation \eqref{swTime} becomes
\begin{equation}
    \label{SwTimeTanh}
    \tanh(\mu s)
    =
    \frac{\mu}{\la}
    \frac{\tr(AB)-\tr(A^2)}{\tr(AB)-\tr(B^2)}
    \tanh(\la t).
\end{equation}
This equation has always a positive solution $s=s(t)$ if 
$
\frac{
    \tr(AB)-\tr(A^2)
}{
 \tr(AB)-\tr(B^2)   
}
>
0
$. 
In particular 
\begin{equation}
    \label{TaylExp}
    s(t)
    =
    \frac{\tr(AB)-\tr(A^2)}{\tr(AB)-\tr(B^2)}
    t
    +
    o(t)
    \quad 
    \text{ as }
    t \to 0.
\end{equation}
\begin{rem}
    The value 
    \begin{equation*}
        s'(0)
        =
        \frac{
            \tr(AB) - \tr(A^2)
        }{
            \tr(AB) - \tr(B^2)
        }
    \end{equation*}
    is intrinsic and depends only on the equation
    \begin{equation*}
        \frac{
            \pa \Phi
        }{
            \pa t        
        }
        (t,s(t))
        =
        \frac{
            \pa \Phi
        }{
            \pa s        
        }
        (t,s(t)).
    \end{equation*}
    Indeed, if we differenciate this equation with respect to $t$, we obtain
    \begin{equation*}
        \frac{
            \pa^2 \Phi
        }{
            \pa t^2        
        }
        (t,s(t))
        +
        \frac{
            \pa^2 \Phi
        }{
            \pa t \pa s        
        }
        (t,s(t))
        s'(t)
        =
        \frac{
            \pa^2 \Phi
        }{
            \pa t \pa s        
        }
        (t,s(t))
        +
        \frac{
            \pa^2 \Phi
        }{
            \pa s^2        
        }
        (t,s(t))
        s'(t).
    \end{equation*}
    Thus
    \begin{equation*}
        s'(0)
        =
        \ddfrac{
            \frac{
                \pa^2 \Phi
            }{
                \pa t \pa s        
            }
            (0,0)
            -
            \frac{
                \pa^2 \Phi
            }{
                \pa t ^2        
            }
            (0,0)
        }
        {
            \frac{
                \pa^2 \Phi
            }{
                \pa t \pa s        
            }
            (0,0)
            -
            \frac{
                \pa^2 \Phi
            }{
                \pa s ^2        
            }
            (0,0)
        }
        =
        \frac{
            \tr(AB) - \tr(A^2)
        }{
            \tr(AB) - \tr(B^2)
        }.
    \end{equation*}
\end{rem}
So, for every $t>0$ we obtained an $s(t)$ such that the \textit{bang-bang} control
\begin{equation*}
    u(\tau)
    =
    \begin{cases}
        1 \quad \text{ if } \tau \in [0,t), \\
        0 \quad \text{ if } \tau \in [t,t+s(t)),
    \end{cases}
\end{equation*}
and extended by periodicity for $\tau \in [t+s(t),+\infty)$,
satisfies all conditions of Pontryagin Maximum Principle. 
\\
Now, it remains just to determine for which $t$ the Principal Lyapunov exponent of the system is maximal. Define the family of controls
\begin{equation*}
    u_{t,s}(\tau)
    =
    \begin{cases}
        1 \quad \text{ if } \tau \in [0,t), \\
        0 \quad \text{ if } \tau \in [t,t+s),
    \end{cases}
    \quad
    t,s>0,
\end{equation*}
and denote with $X_{t,s}$ the solution of \eqref{Group_eq} with control equal to $u_{t,s}$. Let $\alpha(t,s)$ be the greatest eigenvalue of $X_{t,s}(t+s)$, so that the Principal Lyapunov exponent corresponding to the control $u_{t,s}$ is
\begin{equation*}
    \ell(t,s)
    =
    \frac{1}{t+s}
    \log
    \big(
        \alpha(t,s)
    \big).
\end{equation*}
The following result help us in this task.
\begin{lem}
    \label{MonotLyapExp}
    With the notation introduced above:
    \begin{enumerate}
        \item if 
            $
            \tr
            \big(
                [A,B]^2
            \big)
            >
            0,
            $
            then 
            $
            \ell(\frac{t}{2},\frac{s}{2})
            >
            \ell(t,s)
            $;
        \item if 
            $
            \tr
            \big(
                [A,B]^2
            \big)
            <
            0
            $,
            then 
            $
            \ell(\frac{t}{2},\frac{s}{2})
            <
            \ell(t,s)
            $;   
    \end{enumerate}
\end{lem}
We first show how to obtain the estimate for the Principal Lyapunov exponent, then we prove the Lemma. 
\begin{thm}
    If 
            $
            \tr
            \big(
                [A,B]^2
            \big)
            >
            0,
            $
    then for every $t>0$ it holds
    \begin{equation}
        \label{MaxLyapSw}
        \ell(t,s(t))
        <
        \lim _{\tau\to 0}
        \ell(\tau,s(\tau))
        =
        \frac{1}{2}
                \sqrt{
                    \frac{
                            \tr \big([A,B]^2\big)
                        }
                    {
                         2\tr(AB)-\tr(A^2)-\tr(B^2)
                    }
                }
    \end{equation}
    If instead 
            $
            \tr
            \big(
                [A,B]^2
            \big)
            <
            0,
            $
    then control with switches are not optimal.
\end{thm}
\begin{proof}
First, suppose that 
            $
            \tr
            \big(
                [A,B]^2
            \big)
            >
            0
            $.
One can check that this condition implies that the coefficient in \eqref{SwTimeTanh} is smaller than 1:
\begin{equation*}
    \frac{\mu}{\la}
    \frac{\tr(AB)-\tr(A^2)}{\tr(AB)-\tr(B^2)}
    <
    1.
\end{equation*}
So, for every $t>0$ it makes sense to take the inverse hyperbolic tangent in \eqref{SwTimeTanh} on both sides.
Moreover,
from point $(1)$ of Lemma \ref{MonotLyapExp}, for any $t>0$ we have 
\begin{equation*}
    \ell
    \left(
    \frac{t}{2},
    \frac{s(t)}{2}
    \right)
    >
    \ell
    (
    t,s(t)
    ). 
\end{equation*}
In particular, we can find $t_1<t$ such that $t_1+s(t_1)=\frac{1}{2}(t+s(t))$, so that
\begin{equation*}
    \ell(t_1,s(t_1))
    >
    \ell
    \left(
        \frac{t}{2},
        \frac{s(t)}{2}
    \right)
    >
    \ell
    (
        t,s(t)
    ).
\end{equation*}
Moreover, we have that $\lim _{t\to 0} s(t) = 0$, so 
\begin{equation*}
    \ell(t,s(t))
    <
    \lim _{\tau\to 0}
    \ell(\tau,s(\tau)).
\end{equation*}
To evaluate the last limit, one can use Taylor expansions. 
This leads to 
\begin{equation*}
    \lim _{t\to 0}
    \ell(t,s(t))
    =
    \frac{1}{2}
                \sqrt{
                    \frac{
                            \tr \big([A,B]^2\big)
                        }
                    {
                         2\tr(AB)-\tr(A^2)-\tr(B^2)
                    }
                }
    .
\end{equation*}
It is possible to avoid to compute explicitly this limit (see next Subsection). 

Suppose now that 
$
    \tr \big([A,B]^2\big)
    <
    0
$.
Notice that this can happen only if $\tr(A^2),\tr(B^2)>0$. 
Similarly to the previous case, this inequality implies
\begin{equation*}
    \frac{\mu}{\la}
    \frac{\tr(AB)-\tr(A^2)}{\tr(AB)-\tr(B^2)}
    >
    1.
\end{equation*}
So, we can solve equation \eqref{SwTimeTanh} only for $t<\Bar{t}$, where
\begin{equation*}
    \Bar{t}
    =
    \frac{1}{\lambda}
    \arctanh
    \left(
        \frac{\la}{\mu}
        \frac{\tr(AB) - \tr(A^2)}{\tr(AB) - \tr(B^2)}
    \right).
\end{equation*}
Moreover, 
$
\lim_{t\to \Bar{t}} 
s(t) 
= 
+\infty
$.
Then, from point (2) of Lemma \ref{MonotLyapExp}, we have
\begin{equation*}
    \ell(2t,2s(t))
    >
    \ell(t,s(t)).
\end{equation*}
So, in particular we can find $t_1\in(t,\Bar{t})$ such that $t_1+s(t_1)=2(t+s(t))$. So, this implies
\begin{equation*}
    \ell(t_1,s(t_1))
    >
    \ell(2t,2s(t))
    >
    \ell(t,s(t))
    .
\end{equation*}
So, we obtain that 
\begin{equation*}
    \ell(t,s(t))
    <
    \lim _{\tau\to \Bar{t}}
    \ell(\tau,s(\tau)).
\end{equation*}
But as $t\to \Bar{t}$, the trajectory spends more and more time with control equal to $0$. So, this strategy is for sure worse than spending all the time with control equal to 1, which is always admissible if 
$
    \tr \big([A,B]^2\big)
    <
    0
$
. Hence, in this case switches are not optimal. 
\end{proof}
We now prove Lemma \ref{MonotLyapExp}.
\begin{proof}[Proof of Lemma \ref{MonotLyapExp}]
    Recall that 
    \begin{equation*}
        l(t,s)
        =
        \frac{1}{t+s}
        \log
        \big(
            \alpha(t,s)
        \big).
    \end{equation*}
    Hence, we can write 
    \begin{align*}
        \ell
        \left(
            \frac{t}{2},
            \frac{s}{2}
        \right)
        -
        \ell(t,s)
        &=
        \frac{2}{t+s}
        \log
        \left(
            \alpha
        \left(
            \frac{t}{2},
            \frac{s}{2}
        \right)
        \right)
        -
        \frac{1}{t+s}
        \log
        \big(
            \alpha(t,s)
        \big)
        =
        \\
        &=
        \frac{1}{t+s}
        \log
        \left(
            \frac{
                \alpha
                \left(
                    \frac{t}{2},
                    \frac{s}{2}
                \right)^2
            }{
                \alpha(t,s)
            }
        \right).
    \end{align*}
    So, we can see that this difference is positive if and only if
    $$
        \frac{
                \alpha
                \left(
                    \frac{t}{2},
                    \frac{s}{2}
                \right)^2
            }{
                \alpha(t,s)
            }
            >
            1.
    $$
    Since $\alpha(t,s)>1$ for any $t,s>0$, we have that 
    \begin{equation*}
            \alpha
            \left(
                    \frac{t}{2},
                    \frac{s}{2}
            \right)^2
            >
            \alpha(t,s)
            \quad 
            \text{iff}
            \quad
            \alpha
            \left(
                    \frac{t}{2},
                    \frac{s}{2}
            \right)^2
            +
            \frac{1}{\alpha
            \left(
                    \frac{t}{2},
                    \frac{s}{2}
            \right)^2}
            >
            \alpha(t,s)
            +
            \frac{1}{\alpha(t,s)},
    \end{equation*}
    which is the same as
    \begin{equation*}
        \tr
        \left(
            X
            \left(
                \frac{t}{2},\frac{s}{2}
            \right)^2
        \right)
        >
        \tr
        \big(
            X(t,s)
        \big).
    \end{equation*}
    The term in the left hand side can be rewritten as 
    \begin{equation*}
        X
        \left(
            \frac{t}{2},\frac{s}{2}
        \right)^2
        =
        X(t,s)
        +
        e^{\frac{t}{2}A}
        \big[
            e^{\frac{s}{2}B}
            ,
            e^{\frac{t}{2}A}
        \big]
        e^{\frac{s}{2}B}.
    \end{equation*}
    So, in the end, it amounts just to compute 
    $
    \tr
    \left(
    e^{\frac{t}{2}A}
        \big[
            e^{\frac{s}{2}B}
            ,
            e^{\frac{t}{2}A}
        \big]
    e^{\frac{s}{2}B}
    \right)
    $,
    which can be done using formula \eqref{ExpMatr}. One obtains
    \begin{equation*}
    \tr
    \left(
    e^{\frac{t}{2}A}
        \big[
            e^{\frac{s}{2}B}
            ,
            e^{\frac{t}{2}A}
        \big]
    e^{\frac{s}{2}B}
    \right)
    =
    \frac{\sinh(\la \frac{t}{2})^2 \sinh(\mu \frac{s}{2})^2}{2\la^2\mu^2}
    \tr
    \big(
        [A,B]^2
    \big).
    \end{equation*}
    So, if 
    $
    \tr
    \big(
        [A,B]^2
    \big)
    >
    0
    $
    ,
    then we obtained point (1) of the Lemma. If instead 
    $
    \tr
    \big(
        [A,B]^2
    \big)
    <
    0
    $,
    then all inequalities are reversed and we obtain point (2).
    \end{proof}
\begin{rem}
    Notice that this is a completely general result about product of matrix exponentials. Indeed, it holds for a generic control $u_{t,s}$, even if it do not correspond to a Pontryagin extremal.
\end{rem}
\end{subsubsection}

\begin{subsubsection}{Limit trajectory as switching time tends to zero}
    In the previous Subsection we saw that if $\tr([A,B]^2)>0$, then the upper bound for the principal Lyapunov is obtained as the switching time tends to zero. 
    Now, we want to discuss some more details about the limit trajectory as the switching time goes to zero. This will allow us to evaluate the upper bound without computing directly the limit.
    \begin{prop}
    \label{ConvReal}
        The limit trajectory as the switching time goes to zero is the singular trajectory.
    \end{prop}
    \begin{proof}
        First, recall that from \eqref{TaylExp}
        \begin{equation*}
            s(t)
            =
             \frac{\tr(AB)-\tr(A^2)}{\tr(AB)-\tr(B^2)}
            t
            +
            o(t)
            \quad 
            \text{ as }
            t \to 0.
        \end{equation*}
        So, our trajectory spends $t$ time with control equal to 1, which corresponds to matrix $A$, and approximately 
        $
         \frac{\tr(AB)-\tr(A^2)}{\tr(AB)-\tr(B^2)}
        t
        $
        time with control equal to 0, corresponding to matrix $B$.
        For simplicity, let us call 
        $$
        c
        =
          \frac{\tr(AB)-\tr(A^2)}{\tr(AB)-\tr(B^2)},
        $$
        the coefficient appearing in $s(t)$.
        \\
        So, by convex approximation (see Theorem 8.2 in \cite{AgSa}), we know that as the number of switches goes to infinity (i.e., $t$ goes to zero), the flow tends uniformly (in the $C^\infty(\SLD)$-topology) to 
        \begin{equation*}
            T
            \mapsto
            \exp
            \big(
            T
            (
            v A
            +
            (1-v)B
            )
            \big),
        \end{equation*}
        where 
        \begin{equation*}
            v
            =
            \frac{1}{1+c}
            =
            \frac{
                \tr(AB)-\tr(B^2)
            }{
                2\tr(AB)-\tr(A^2)-\tr(B^2)
            }
            ,
        \end{equation*}
        which corresponds to the value of the singular control (see Section 5). 
    \end{proof}
    So, since the convergence in the previous proof is uniform on compact time intervals and since in the case of periodic controls we can reduce to the finite time case (see \ref{perLyap}), we obtain the following Corollary. 
    \begin{cor}
        With the notation used in Subsection 4.4.2, we have
        \begin{equation*}
             \lim _{\tau\to 0}
             \ell(\tau,s(\tau))
             =
             \la_*
             =
             \frac{1}{2}        
            \sqrt{
                    \frac{
                    \tr
                    \big(
                        [A,B]^2
                    \big)
                }{
                    2\tr(AB)
                    -\tr(A^2)
                    -\tr(B^2)
                }
            },
        \end{equation*}
        where $\la_*$ is the positive eigenvalue of \ref{SingVel} (see Section 5).
    \end{cor}
\end{subsubsection}
\end{subsection}
    
\end{section}

\begin{section}{Singular extremals}
        In this Section we are going to deal with the case of singular extremals (see Definition \ref{RegSingExtr}). Recall the definition of the (left-invariant) Hamiltonian function
    \begin{equation*}
        H(\e,u)
        =
        \tr
        \big(
            \e
            (
                uA+(1-u)B
            )
        \big),
    \end{equation*}
    and the switching function
    \begin{equation}
        \vp(t) 
        =
        \tr
        \big(
            \e(t)
            (
                A-B
            )
        \big).
    \end{equation}
    Recall also that the singular extremals correspond to the case of $\vp(t) = 0$ for $t\in[\tau_1,\tau_2]$, and $\tau_1,\tau_2>0$.
    In this case, since $\vp$ is Lipschitz, we can derive it:
    \begin{equation*}
        0
        =
        \dvp(t)
        =
        \tr
        \big(
            \de (t)
            (A-B)
        \big)
        =
        \tr
        \Big(
            \big[
            \e (t)
            ,
            uA + (1-u)B
            \big]
            (A-B)
        \Big)
        =
        \tr
        \big(
            \e (t)
            [
            A , B
            ]
        \big),
    \end{equation*}
    where in the last equality we used the identity 
    $
    \tr([M_1,M_2]M_3)
    =
    \tr(M_1[M_2,M_3]).
    $
    Notice that $\dvp$ does not depend on $u$ and in particular it is again a Lipschitz function.
    So, if $\e_*$ is a point of a singular trajectory, it must satisfy the two following conditions:
    \begin{align}
    \label{SingEq1}
    \left\{
        \begin{array}{l}
            \tr
            \big(
                \e_*
                (
                    A-B
                )
            \big)=0,
            \\
            \tr
            \big(
                \e_*
                [
                A , B
                ]
            \big)=0
        \end{array}
        \right.
        .
    \end{align}
    These two conditions are linear in $\e_*$, so, since $A,B$ and $[A,B]$ are linearly independent, they determine a straight line in $\sld$.
    \\
    If we write 
    $
    \e_*
    =
    \al A + \be B + \gamma [A,B],
    $
    then the second equation implies $\gamma=0$. Hence, from \eqref{CoeffEta}, we see that if we take $c_*=\tr(\e_* A)$, the line $r_{c_*}$ is tangent to the hyperboloid $H$ (see discussion in Subsection 4.2). 
    \\[5pt]
    There are two intersection between the line \eqref{SingEq1} and the hyperboloid $H$ described in Section 2 if and only if 
    \begin{equation}
    \label{SingIntersezIneq}
        \big(
            2\tr(AB)
            -
            \tr(A^2)
            -
            \tr(B^2)
        \big)
        \tr([A,B]^2)
        >
        0.
    \end{equation}
    If there are no intersection, then there are no singular extremals. So, from now on we will assume that this inequality hold.
    \\
    We can compute explicitly these two intersection: let us call $\e_*$ one of  the two intersection point. Then, from \eqref{swCoord} and $\gamma=0$ we obtain
    \begin{equation*}
        \e_*
        =
        \al_*
        A
        +
        \al_*
        \frac{
            \tr(AB)-\tr(A^2)
        }{
            \tr(AB)-\tr(B^2)
        }
        B,
    \end{equation*}
    where $\al_*\in\R$ is chosen such that $\tr(\e_* ^2)=1$. 
    Since the intersection between \eqref{SingEq1} and $H$ is a discrete set and since the solution $\e_*$ is Lipschitz, the only possibility for $\e_*$ is to be constant.  
    Hence
    \begin{equation*}
        0
        =
        \de_*(t)
        =
        [
            \e_*
            ,
            u_* A
            +
            (1-u_*)B
        ],
    \end{equation*}
    from which we can determine the value of the singular control $u_*$
    \begin{equation}
        u_*
        =
        \frac{
            \tr(AB)-\tr(B^2)
        }{
            2\tr(AB)-\tr(A^2)-\tr(B^2)
        }.
    \end{equation}
    In order to be an admissible control, it must hold $u_*\in[0,1]$.
    The inequality $u_*\geq 0$ is equivalent to the condition
    \begin{equation*}
        \left\{
            \tr(AB)
            >
            \frac{\tr(A^2)+\tr(B^2)}{2}
        \right\}
        \cup 
        \left\{
            \tr(AB)
            \leq
            \tr(B^2)
        \right\},
    \end{equation*}
    while the other inequality $u_*\leq 1$ is equivalent to
    \begin{equation*}
        \left\{
            \tr(AB) 
            \geq
            \tr(A^2)
        \right\}
        \cup 
        \left\{
            \tr(AB)
            <
            \frac{\tr(A^2)+\tr(B^2)}{2}
        \right\},
    \end{equation*}
    and recalling that $\tr(A^2)>\tr(B^2)$, we obtain 
    \begin{prop}
    $u_*$ is an admissible control if and only if 
    \begin{equation*}
            \tr(AB) 
            \geq
            \tr(A^2)
            \quad
        \text{ or }
        \quad 
            \tr(AB)
            \leq
            \tr(B^2).  
    \end{equation*}
    \end{prop}
    Then, we obtain
    \begin{equation}
        \label{SingVel}
        u_*A
        +
        (1-u_*)
        B
        =
        \frac{
                \tr(AB)-\tr(B^2)
            }
            {
                2\tr(AB)-\tr(A^2)-\tr(B^2)
            }
        A
        +
        \frac{
                \tr(AB)-\tr(A^2)
            }
            {
                2\tr(AB)-\tr(A^2)-\tr(B^2)
            }
        B,
    \end{equation}
    and
    \begin{equation}
    \label{trSing}
        \tr
        \Big(
            \big(
                u_*
                A
                +
                (1-u_*)
                B
            \big)^2        
        \Big)
        =
        \frac{1}{2}
        \frac{
            \tr
            \big(
                [A,B]^2
            \big)
        }{
            2\tr(AB)
            -\tr(A^2)
            -\tr(B^2)   
        }.
    \end{equation}
    So, if we take the constant control equal to $u_*$, then the associated Lyapunov exponent is 
    \begin{equation}
    \label{SingLyap}
        l
        =
        \frac{1}{2}        
        \sqrt{
                \frac{
                \tr
                \big(
                    [A,B]^2
                \big)
            }{
                2\tr(AB)
                -\tr(A^2)
                -\tr(B^2)
            }
        }.
    \end{equation}
    A direct computation shows that
    \begin{equation}
        \tr
        \Big(
            \big(
                u_*
                A
                +
                (1-u_*)
                B
            \big)^2        
        \Big)
        \geq
        \tr(A^2)
        \quad 
        \text{iff}
        \quad 
        \tr(AB)
        \geq
        \frac{\tr(A^2)+\tr(B^2)}{2}.
    \end{equation}
    So, taking into account that singular control is admissible if $\tr(AB)\geq \tr(A^2)$, we obtain that singular constant control are better than constant control if $\tr(AB)\geq \tr(A^2)$. 
    On the other hand, if
    $\tr(AB)\leq \tr(B^2)$
    then constant control is better than constant singular control. 
    \\
    The following result ends the discussion about the singular case.
    \begin{prop}
        There are no solution to the adjoint system \eqref{adj_syst} containing a \emph{bang-bang} piece and a singular piece.
    \end{prop}
    \begin{proof}
        We already noticed that for $\e_*$ we have $\gamma=0$ in \eqref{CoeffEta}
        and the intersection line between the planes $\Pi_{A,c_*}$ 
        and $\Pi_{B,c_*}$ is tangent to the hyperboloid $H$, where $c_*=\tr(A\e_*)$. 
        Recall from Section 4 that as the switching time in a \textit{bang-bang} trajectory in $\SLD$ tends to zero, the resulting trajectory tends to a singular trajectory. The convergence in $\SLD$ implies that also the sequence of the corresponding adjoint trajectories tends to the singular adjoint trajectory, which is constant. Thus, in a neighborhood of $\e_*$ sufficiently small and contained in the hyperboloid $H$, the phase portrait is as in Figure \ref{RitrFaseSing}.
        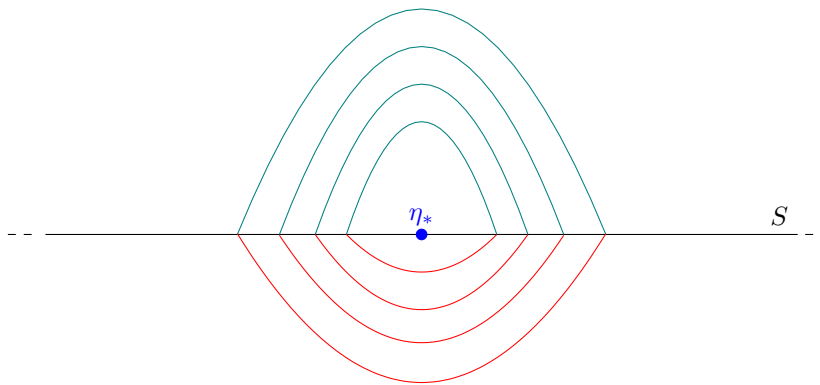
\begin{figure}[ht]
            \centering
            \begin{tikzpicture}
    \draw[-] (-5,0) -- (5,0) 
    node[anchor=south east] {$S$}
    ;
    \draw[dashed] (5.1,0) -- (5.5,0);
    \draw[dashed] (-5.5,0) -- (-5.1,0);
    \filldraw[blue] (0,0) circle (2pt) node[anchor=south]{$\e_*$};
    \draw[color=teal,domain=-1:1] plot (\x,{-1.5*(\x)^2+1.5}) ;
    \draw[color=teal,domain=-1.414:1.414] plot (\x,{-(\x)^2+2}) ;
    \draw[color=teal,domain=-1.89:1.89] plot (\x,{-0.7*(\x)^2+2.5}) ;
    \draw[color=teal,domain=-2.45:2.45] plot (\x,{-0.5*(\x)^2+3}) ;
    \draw[color=red,domain=-1:1,smooth] plot (\x, {0.5*(\x)^2-0.5}) ;
    \draw[color=red,domain=-1.414:1.414,smooth] plot (\x, {0.5*(\x)^2-1}) ;
    \draw[color=red,domain=-1.89:1.89,smooth] plot (\x,{0.4*(\x)^2-1.44}) ;
    \draw[color=red,domain=-2.45:2.45,smooth] plot (\x,{0.33*(\x)^2-1.97}) ;
\end{tikzpicture}
            \caption{Phase portrait on the hyperboloid $H$ for the adjoint system \eqref{adj_syst} near $\e_*$. The black line $S$ is the switching curve, $\e_*$ is the singular point. Above and below $S$ we have the two \textit{bang} pieces, and as the switching time goes to zero the adjoint trajectory tends to $\e_*$. }
            \label{RitrFaseSing}
        \end{figure}
        \\
        Thus, any possible trajectory starting from $\e_*$ and escaping the switching curve $S$ would inevitably cross some of the periodic trajectories tending to $\e_*$, contradicting the local uniqueness of the solution of an ODE. 
    \end{proof}
    So, to resume, we found that if an extremal contains a singular piece, then the whole extremal is singular, and the corresponding Lyapunov exponent is 
    \begin{equation}
        \ell_{\text{sing}}(A,B)
        =
        \frac{1}{2}        
        \sqrt{
                \frac{
                \tr
                \big(
                    [A,B]^2
                \big)
            }{
                2\tr(AB)
                -\tr(A^2)
                -\tr(B^2)
            }
        },
    \end{equation}
    and this is the optimal solution if $\tr(AB)\geq \tr(A^2)$.
\end{section}

\begin{section}{Discussion of cases}
    We end up with some final remarks about which is the optimal strategy depending on the values $\tr(A^2)$,$\tr(B^2)$,$\tr(AB)$. As before, we always assume $\tr(A^2)\geq \tr(B^2)$. 
    \\
    So far, we have obtained that all possible extremals satisfying PMP are in one of the following form:
    \begin{itemize}
        \item constant control;
        \item constant singular control;
        \item if $\tr(A^2)<0$, \emph{bang-bang} periodic control as in Subsection 4.4.1.
    \end{itemize}
    As pointed out in Subsection 4.4.2, there is also this other kind of extremals accumulating at singular extremals. However, we saw that their associated principal Lyapunov exponent is at most strictly less then the one associated to singular extremal, so we can omit them from the discussion.  
    \\
    First, assume that $\tr(A^2)\geq 0$. In this case we already saw in Section 5 that if $\tr(AB) \geq \tr(A^2)$, then the value of the Principal Lyapunov exponent given by the constant singular control is bigger than the one given by the constant control. So, the first two point of the Main Theorem are proved. 
    \\
    We turn now to the $\tr(A^2)<0$ case. We have the following restriction on the possible values of $\tr(AB)$.
    \begin{lem}
        If 
        $
        \tr(B^2)
        \leq
        \tr(A^2)
        <
        0
        $, 
        then 
        $
        |\tr(AB)|
        \geq
        \sqrt{
           \tr(A^2)\tr(B^2)
        }
        $. 
        In particular 
        $
        \tr([A,B]^2)
        \geq 
        0
        $. 
    \end{lem}
    \begin{proof}
        Up to a change of base, we can suppose
        \begin{equation*}
            A
            =
            \begin{pmatrix}
                0 & -\la \\
                \la & 0
            \end{pmatrix}
            ,
            \quad
            B
            =
            P
            \begin{pmatrix}
             0 & -\mu \\
             \mu & 0
            \end{pmatrix}
            P^{-1},
            \quad 
            P
            =
            \begin{pmatrix}
                a & b \\
                c & d
            \end{pmatrix}\in \SLD,
            \quad 
            \la,\mu\in\R.
        \end{equation*}
        Since $A$ in this form commutes with rotations, that is matrix of the form
        \begin{equation*}
            R_\theta
            =
            \begin{pmatrix}
                \cos\theta & -\sin\theta \\
                \sin\theta & \cos\theta
            \end{pmatrix}
            ,
            \quad 
            \theta\in[0,2\pi],
        \end{equation*}
        we can use another change of base given by a rotation matrix to kill one element of $P$. That is, we can further assume that $P$ is in the form
        \begin{equation*}
            P
            =
            \begin{pmatrix}
                \al & \be \\
                0 & \frac{1}{\al}
            \end{pmatrix}
            ,
            \quad 
            \al,\be\in\R,
            \quad
            \al\neq 0.
        \end{equation*}
        With $A,B$ in this form, we can compute directly $\tr(AB):$
        \begin{equation*}
            \tr(AB)
            =
            -\la\mu
            \left(
                \frac{1}{\al^2}
                +
                \al^2
                +
                \be^2
            \right).
        \end{equation*}
        Since 
        $
            \frac{1}{\al^2}
            +
            \al^2
            \geq 
            2
        $
        ,
        we obtain 
        \begin{equation*}
            |\tr(AB)|
            \geq 
            2|\la\mu|
            =
            \sqrt{
                \tr(A^2)\tr(B^2)
            }.
        \end{equation*}
    \end{proof}

    \begin{rem}
        From the previous proof, we can give a geometric interpretation of the sign of $\tr(AB)$ in the case $\tr(B^2)\leq \tr(A^2)<0$. 
        Define the cone of imaginary matrices in $\sld$:
        \begin{equation*}
            \mathcal{C}
            =
            \{
                M\in\sld 
                \: | \:
                \tr(M^2)<0
            \}.
        \end{equation*}
        To see that $\mathcal{C}$ is a cone, one can reason as at the beginning of Subsection 3.1. Using the form for $A,B$ introduced in the previous proof, if $\tr(AB)<0$, then $\la\mu>0$, that is $A,B$ are in the same connected component. If instead $\tr(AB)>0$, then $\la\mu<0$, hence $A,B$ are in two different connected component.
    \end{rem}
    In particular, if $\tr(AB)<0$, then the whole segment of velocities 
    \begin{equation*}
        \{
            uA+(1-u)B
            \:|\:
            u\in[0,1]
        \}
    \end{equation*}
    is inside the cone $\mathcal{C}$. In this case, if the singular control is admissible, then it has imaginary eigenvalues, hence singular constant control cannot be optimal (the resulting trajectory in $\SLD$ would be periodic, hence bounded). So, in this case the best admissible strategy is the one given in Subsection 4.4.1, and we obtain the third point of Main Theorem.
    \\
    If instead $\tr(AB)>0$, then part of the segment of velocities is outside the cone, and in particular singular control is admissible since $\tr(A^2)<0$.
    \\
    The following Lemma shows that in this case constant singular control are again optimal.
    \begin{lem}
        If $\tr(B^2)\leq\tr(A^2)<0$ 
        and 
        $
            \tr(AB)
            \geq
            \sqrt{\tr(A^2)\tr(B^2)}
        $, 
        then 
            \begin{align*} 
                    \frac{2}{\pi}
                    \frac{
                        1
                    }{
                        \sqrt{\frac{-2}{\tr(A^2)}}
                        +
                        \sqrt{\frac{-2}{\tr(B^2)}}
                    }
                    \arcosh
                    \left(
                        \frac{\tr(AB)}{\sqrt{\tr(A^2)\tr(B^2)}}
                    \right)
                    \leq
                        \sqrt{
                            \frac{1}{2}
                            \frac{
                                    \tr (AB)^2-\tr(A^2)\tr(B^2)
                                }
                            {
                                 2\tr(AB)-\tr(A^2)-\tr(B^2)
                            }
                        }
                    .
            \end{align*}
    \end{lem}
    \begin{proof}
        First, notice that the two quantities are homogeneus in $\tr(A^2),\tr(B^2),\tr(AB)$. So, left hand side can be rewritten as
        \begin{equation*}
            \frac{2}{\pi}
            a
            \frac{\sqrt{b}}{1+\sqrt{b}}
            \arcosh
            \left(
                \frac{-c}{\sqrt{b}}
            \right),
        \end{equation*}
        where $a=\sqrt{\frac{\tr(A^2)}{-2}}$, $b=\frac{\tr(B^2)}{\tr(A^2)}$, $c=\frac{\tr(AB)}{\tr(A^2)}$. From the hypothesis of the Lemma, we deduce $a>0$, $b\geq 1$, $c\leq-\sqrt{b}\leq -1$. Using similar notation, the right hand side of the inequality in the thesis is 
        \begin{equation*}
            a 
            \sqrt{
                \frac{
                    c^2 - b 
                }{
                    b+1-2c
                }
            }.
        \end{equation*}
        Define
        \begin{align*}
            f(b,c)
            &=
            \frac{2}{\pi}
            \frac{\sqrt{b}}{1+\sqrt{b}}
            \arcosh
            \left(
                \frac{-c}{\sqrt{b}}
            \right),
            \\
            g(b,c)
            &=
            \sqrt{
                \frac{
                    c^2 - b 
                }{
                    b+1-2c
                }
            }.
        \end{align*}
    Notice that for every $b$, if $c\to -\sqrt{b}$ then both functions tend to zero, i.e.:
    \begin{equation*}
        \forall b \geq 1
        \quad 
        \lim_{c\to -\sqrt{b}}
        f(b,c)
        =
        0
        =
        \lim_{c\to -\sqrt{b}}
        g(b,c).
    \end{equation*}
    Hence, if for every $b\geq 1$ and every $c\leq -\sqrt{b}$ we have 
    \begin{equation}
        \label{ineqineq}
        \frac{\pa f}{\pa c} (b,c)
        \geq
        \frac{\pa g}{\pa c} (b,c),
    \end{equation}
    then we obtain $f(b,c)\leq g(b,c)$ for every $b\geq 1$ and $c\leq -\sqrt{b}$.
    These derivatives are 
    \begin{align*}
        \frac{\pa g}{\pa c} (b,c)
        &=
        -\frac{2}{\pi}
        \frac{\sqrt{b}}{\sqrt{b}+1}
        \frac{1}{\sqrt{c^2-b}}
        \\
        \frac{\pa g}{\pa c} (b,c)
        &=
        \frac{1}{\sqrt{c^2-b}}
        \frac{
            (c-1)(b-c)
        }{
            (b+1-2c)^{\frac{3}{2}}
        }.
    \end{align*}
    So \eqref{ineqineq} reduces to
    \begin{equation*}
        \frac{2}{\pi}
        \frac{\sqrt{b}}{\sqrt{b}+1}
        \leq 
         \frac{
            (1-c)(b-c)
        }{
            (b+1-2c)^{\frac{3}{2}}
        }.
    \end{equation*}
    Since we want this inequality to hold for every $c\leq -\sqrt{b}$ and since the right hand side is decreasing in $c$, it is sufficient to prove
    \begin{equation*}
        \frac{2}{\pi}
        \frac{\sqrt{b}}{\sqrt{b}+1}
        \leq 
         \frac{
            (1+\sqrt{b})(b+\sqrt{b})
        }{
            (b+1+2\sqrt{b})^{\frac{3}{2}}
        }.
    \end{equation*}
    After a few simplifications, this reduces to 
    \begin{equation*}
        (\sqrt{b}+1)^{\frac{3}{2}}
        \geq
        \frac{2}{\pi},
    \end{equation*}
    which is true for every $b\geq 1$. 
    \end{proof}
    This concludes the proof of the fourth case of Main Theorem.
\end{section}

\appendix

\begin{section}{Pontryagin Maximum Principle}
    In this paper we used a version of Pontryagin Maximum Principle (PMP) which is slightly different from the one that is usually written in books (see, for instance \cite{AgSa} or \cite{AgBaBo}).
    So, for completeness, we recall here the precise statement of the PMP that we used in Section 2.
    \\
    Let $M$ be a smooth manifold of dimension $n$, 
    $U\subset \R^m$, and $(f_u)_{u\in U}$, a family of smooth vector field on $M$. Suppose furthermore that 
    $
    M
    \times 
    \overline{U} 
    \ni(q,u)
    \mapsto 
    f_u(q)
    $ 
    and 
    $
    M
    \times 
    \overline{U} 
    \ni(q,u)
    \mapsto 
    \frac{
        \partial f_u
    }{
        \partial q
    }(q)
    $ 
    are continuous.
    \\ 
    Fix $T>0$ and define $\U=\{u:[0,T]\to U, u\in L_{\textit{loc}}^\infty([0,T],U) \}$ the space of admissible controls. Fix $q_0\in M$ and for all $u\in\U$ denote with $q_u : [0,T] \to M$ the solution to the Cauchy Problem 
    \begin{equation}
	\label{contr_syst}
	\begin{cases}
		\dot{q}_u(t)=f_{u(t)}(q_u(t))\\
		q_u(0)=q_0
	\end{cases}.    
    \end{equation}
    Existence and uniqueness of solution to \eqref{contr_syst} are guaranteed by Carathéodory Theorem, see, for instance \cite{BrePi}.
    \\
    Let $a: M \to \R$ be a smooth function. We define the cost functional 
    \begin{equation*}
	J(u)=a(q_u(T)), \quad u\in\U.    
    \end{equation*}
    Consider the following optimal control problem:
    \begin{prob}[Optimal control]
	Maximize $J$ among all admissible control $u$, i.e., find
        $\Tilde{u}\in\U$ 
        such that 
	\begin{equation}
		\label{opt_contr}
		\begin{cases}
				\dot{q}_{\Tilde{u}}(t)
                =
                f_{\Tilde{u}(t)}(q_{\Tilde{u}}(t)), 
                \quad 
                t\in[0,T]
                \\
			q_{\Tilde{u}}(0)
                =
                q_0 \\
				a(q_{\Tilde{u}}(T))=\displaystyle\max_{u\in\U}a(q_{u}(T))
		\end{cases}.    
	\end{equation}
	\end{prob}
    A solution $\tu\in\U$ to this problem is called \emph{optimal control} and the corresponding solution to \eqref{contr_syst} is called \emph{optimal trajectory}. 
    \begin{defn}[Attainable set]
				Fix $q_0\in M$. The set of attainable points from $q_0$ of the control system \eqref{contr_syst} at time $T>0$ from $q_0\in M$ is 
				\begin{equation*}
				    \A_{q_0}
                        = 
                        \{
                        q\in M 
                        \; | \; 
                        \exists 
                        u \in \U 
                        \text{ such that } 
                        q_u(T)=q
                        \}.
				\end{equation*} 
    \end{defn} 
    \begin{rem}
	Let $q_{\tu}$ be an optimal trajectory for the optimal control problem. 
        Then $a(q_{\tu}(T))\in \partial a(\A_{q_0})$.
    \end{rem}
    \begin{thm}[Pontryagin Maximum Principle]
	\label{PMP}
	Let $\Tilde{u}\in \U$ be a control such that the point $q_{\tu}(T)$ is a local maximum point for the functional $a$.
        Define $h_u(\la)=\langle \la, f_u(\pi(\la))\rangle$, for $u\in U$ and $\la\in T^*M$. Then, there exists a Lipschitzian curve $\la : [0,T] \to T^*M$ such that
	\begin{align}
            &\la(t) 
            \neq 
            0, 
            \\
		\label{hamilt_sistem}
		&\dot{\la}(t) 
            = 
            \Vec{h}_{\Tilde{u}(t)}(\la(t)), \\
		\label{maxim_cond}
		&h_{\Tilde{u}(t)}(\la(t))
            =
            \max _{u\in U} 
            h_u(\la_t),
	\end{align}
	for almost all $t\in [0,T]$, and moreover
	\begin{equation}
		\label{trans_condition}
		\la_T
            =
            \D_{q_{\tu}(T)} a.    
	\end{equation}
    \end{thm}
    We will call \eqref{hamilt_sistem} the Hamiltonian system associated to Problem \eqref{opt_contr}, \eqref{maxim_cond} the maximal condition and \eqref{trans_condition} transversal condition.
    \\
    We recall now some results that we used in Section 2 about Hamiltonian systems on Lie groups. For more detailed reference see Chapter 18 in \cite{AgSa} and Section 7.6 in \cite{AgBaBo}.
    \\
    Let $M\subset GL(N)$ be a Lie Group and $\mathcal{M}$ its Lie algebra.
    Given $q\in M$, define $L_q(p)=qp$, $p\in M$, the multiplication on the left by $q$.
    \\
    Take a Hamiltonian function $h\in C^\infty(T^*M)$ and define
    \begin{equation*}
        \mathcal{H}(\xi,q)
	=
	h(L_{q^{-1}}^*\xi,q), 
        \quad 
        \xi\in T_q^*M, 
        q\in M.
    \end{equation*}
    We will call $\Ham$ the trivialized Hamiltonian function of $h$. 
    We say that $h$ is left-invariant if $\mathcal{H}$ does not depend on $q$. 
    In this case, the function $\Ham$ can be seen simply as a smooth function on $\calM^*$, 
    hence its differential $\D\Ham$ can be seen as a function from $\calM^*$ in $(\calM ^*)^*=\calM$.
    \\
    Recall the definition of the adjoint representation: given $v\in\calM$, $\mathrm{ad} v (w)= [v,w]$, where $w\in\calM$.
    \begin{prop}
    \label{LieGroup1}
	Let $M$ be a Lie group and take $h\in C^\infty(T^*M)$. If $h$ is left-invariant, then the associated Hamiltonian system can be rewritten as
	\begin{equation}
		\label{HamSystLieG}
		\begin{cases}
		      \dot{q}=L_{q*}\D\mathcal{H}, \\
		      \dot{\xi}=(\mathrm{ad}\;\D\mathcal{H})^*\xi.
		\end{cases}
	\end{equation}
    \end{prop}
    In our particular case, that is $M=\SLD$ and $\calM=\sld$, we have that the Killing bilinear form on $\sld$ is non-degenerate, so it provides an intrinsic identification between $\calM$ and $\calM^*$.
    That is, for every $\xi\in \calM^*$ there is an $\e\in\calM$ such that 
    \begin{equation*}
	K(\e,v)
        =
        \langle 
            \xi, v
        \rangle 
        \quad 
        \forall 
        v\in \calM. 
    \end{equation*}
    So, one can compute
    \begin{multline*}
	K(\dot{\e},v)
	=
	\langle
		\dot{\xi},v
	\rangle
	=
	\langle
		(\mathrm{ad}\;\D\mathcal{H})^*\xi,v
	\rangle
	=
	\langle
		\xi,(\mathrm{ad}\;\D\mathcal{H})v
	\rangle
	=
	\\
	=
	\langle
		\xi,[\D\mathcal{H},v]
	\rangle
	=
	K(\e,[\D\mathcal{H},v])
	=
	K([\e,\D\mathcal{H}],v),
	\quad 
	\forall v \in T_{\mathbbm{1}}M.
    \end{multline*}
    So, the final form for the Hamiltonian system is 
	\begin{equation}
		\label{HamSystLieG2}
		\begin{cases}
				\dot{q}=L_{q*}\D\mathcal{H}, \\
			\dot{\e}=[\e,\D\mathcal{H}].
		\end{cases}
	\end{equation}
    Notice that the second equation does not involve $q$.
    \\
    Finally, concerning transversal condition for the Optimal Control Problem \eqref{max_prob}, we have the following result.
    \begin{lem}
	\label{lemma_transv}
	Let $\tr : \SLD \to \R$ be the trace operator on $\SLD$.
        Let us identify the solutions 
        $\xi(\cdot)\in\sld^*$
        and 
        $\e(\cdot)\in\sld$
        via the Killing form of $\sld$, that is 
        $
        \xi(t)
        =
        K(\e(t),\cdot)
        $. 
        Then 
	$$
	   \e(T)
            = 
            \frac{1}{4} 
            \left(
                X_{\tu}(T)
                -
                \frac{\tr(X_{\tu}(T))}{2}I
            \right).
	$$
    \end{lem}
    \begin{proof}
	Define $X=X_{\tu}(T)$ and let $U\in\sld$. Then
	\begin{align*}
		\langle
			\xi(T)
				,
			X U
		\rangle 
		=
		\langle
			\mathrm{d}_X \tr, XU 
		\rangle 
		=
		\frac{\D}{\D t}
		\left(
			\tr(
    				X e^{tU}
                )
		\right)_{|t=0}
            =
	    \tr(XU).
	\end{align*}
	Moreover 
        $\tr(XU)=\tr((X-\frac{\tr(X)}{2}I)U)$ 
        and 
        $X-\frac{\tr(X)}{2}I\in\sld$. 
        Recalling that 
        $K(U,V)=4\tr(VU)$ 
        for 
        $U,V\in\sld$ 
        and that $K$ is non degenerate bilinear form, 
        it follows that there exists a 
        $Y\in\sld$ such that 
        $
        K(Y,U)
        =
        \langle \D _X \tr , XU \rangle
        $ 
        for all $U\in\sld$. 
        Then 
	$$
		4\tr(YU)
            =
            \tr
            \left(
                \left(
                    X
                    -
                    \frac{\tr(X)}{2}I
                \right)U
            \right) 
            \quad 
            \text{for all }
            U\in\sld.
	$$
	Hence 
        $
        Y
        =
        \frac{1}{4}
        \left(
        X
        -
        \frac{\tr(X)}{2}
        I
        \right)
        $, 
        as we wished.
    \end{proof}
\end{section}

\newpage

\printbibliography

\end{document}